\documentclass[11pt]{amsart}
\makeatletter
\usepackage{graphicx} 
\usepackage{geometry}
\usepackage{amsmath,amssymb,amstext,mathrsfs,stmaryrd,calligra}
\usepackage{tikz-cd}
\usepackage{comment}
\newgeometry{margin=1.0in}

\usepackage[colorlinks=true,linkcolor=blue,citecolor=blue,pagebackref
]{hyperref}
\usepackage[hyperpageref]{backref}

\numberwithin{equation}{section}
\newtheorem{thm}{Theorem}[section]
\newtheorem{lm}[thm]{Lemma}
\newtheorem{prop}[thm]{Proposition}
\newtheorem{cor}[thm]{Corollary}
\newtheorem{conj}[thm]{Conjecture}
\theoremstyle{definition}
\newtheorem{ass}[thm]{Assumption}
\theoremstyle{definition}
\newtheorem{defi}[thm]{Definition}
\theoremstyle{remark}
\newtheorem{rmk}[thm]{Remark}
\theoremstyle{remark}
\newtheorem{ex}[thm]{Example}
\theoremstyle{remark}
\newtheorem{warning}[thm]{Warning}
\theoremstyle{Theorem}
\newtheorem{introthm}{Theorem}

\newcommand{\ra}{\rightarrow}
\newcommand\oo{\mathcal{O}}
\newcommand{\xx}{\mathcal{X}}
\newcommand{\dd}{\mathcal{D}}
\newcommand{\mm}{\mathcal{M}}

\newcommand\C{\mathbb C}
\newcommand{\Z}{\mathbb Z}
\newcommand{\Q}{\mathbb{Q}}
\newcommand{\R}{\mathbb R}
\newcommand{\mmm}{\mathcal{M}}
\newcommand\aaa{\mathcal{A}}
\newcommand{\cal}{\mathcal}

\newcommand{\ssigma}{\underline{\sigma}}
\newcommand{\cHom}{\mathcal{H}\!{\it om}}

\DeclareMathOperator{\stab}{Stab}

\DeclareMathOperator{\im}{Im}

\DeclareMathOperator{\Hom}{Hom}

\DeclareMathOperator{\id}{id}

\DeclareMathOperator{\quot}{Quot}
\DeclareMathOperator{\sch}{Sch}
\DeclareMathOperator{\set}{Set}
\DeclareMathOperator{\dga}{dgA}
\DeclareMathOperator{\spec}{Spec}
\DeclareMathOperator{\gpd}{Gpd}
\DeclareMathOperator{\map}{Map}
\DeclareMathOperator{\ind}{Ind}
\DeclareMathOperator{\Mod}{Mod}
\DeclareMathOperator{\sym}{Sym}
\DeclareMathOperator{\Ann}{Ann}
\DeclareMathOperator{\Ind}{Ind}
\DeclareMathOperator{\alg}{-Alg}
\DeclareMathOperator{\Perf}{Perf}
\DeclareMathOperator{\Cat}{Cat}
\DeclareMathOperator{\Fun}{Fun}
\DeclareMathOperator{\HNP}{HNP}
\DeclareMathOperator{\HN}{HN}
\DeclareMathOperator{\precat}{PrCat}
\DeclareMathOperator{\coker}{coker}
\DeclareMathOperator{\filt}{Filt}
\DeclareMathOperator{\cone}{Cone}
\newcommand{\hhom}{\mathcal{H}\!{\it om}}
\mathchardef\mhyphen="2D
\title{Deformations of locally constant stability conditions and good moduli spaces}
\author{Ian Selvaggi}
\address{Scuola Internazionale Superiore di Studi Avanzati (SISSA), Italy} %
\email{iselvagg@sissa.it}
\date{}

\begin{document}

\begin{abstract}
    We give a structure result on the set of locally constant stability conditions, $\stab(\dd/R)$, defined by Halpern-Leistner--Robotis \cite{halpernleistner2025spaceaugmentedstabilityconditions} showing that it has the structure of a complex manifold, in total analogy with Bridgeland's work. As a consequence, we show that the property of having relative mass-hom bounds and the existence of good moduli spaces depends only on the connected components of $\stab(\dd/R)$. Lastly, we observe that the datum of a locally constant stability condition is equivalent to that of a flat family of stability conditions, as described in \cite{bayer2021stability}, in the context of noncommutative algebraic geometry.
\end{abstract}
\maketitle
\tableofcontents 
\section{Introduction}
Since the introduction of stability conditions by Bridgeland in \cite{bridgeland2007stability}, a major problem in the field has been the construction of moduli spaces of semistable objects. Needless to say, a great deal of very deep mathematics has emerged from it. In a first approximation, there are two natural questions to ask given a triangulated category $\dd$ together with a stability condition $\sigma\in\stab(\dd)$:
\begin{enumerate}
        \item\label{questionEX} Under which hypotheses on $\sigma$ is the substack of semistable objects $\mathcal{M}_{\sigma}^{ss}(v)\hookrightarrow\mathcal{M}_\sigma{(\phi,\phi+1]}$ algebraic and admits a proper good moduli space $M_\sigma^{ss}(v)$?
        \item\label{questionDEF} Which of these conditions are preserved as $\sigma\in\stab(\mathcal{D})$ varies? 
    \end{enumerate}
    It turns out that even defining the moduli functor, a question tackled first in \cite{AbramovichPolishchuk+2006+89+130} and later in \cite{TODA_k3}, is quite challenging and nontrivial. To complicate things, unlike the case of coherent sheaves, these moduli stacks do not admit any meaningful group action, making the techniques from GIT unavailable in this context. As a result, the best approach so far has been considering the problem as that of building moduli of objects in abelian categories, a topic first considered in \cite{artin2001abstract}, and applying the far reaching techniques of \emph{intrinsic GIT}. For the latter, the recent breakthrough \cite{alper2023existence} gives necessary and sufficient conditions for the existence of moduli spaces. 
    
    \begin{thm}[\cite{alper2023existence}, Theorem 7.23]\label{thmexistence}
     Let $k$ be an excellent ring of characteristic 0 and $\aaa$ be a cocomplete $k$-linear abelian category that is locally noetherian. Assume that $\mm_\aaa$, the stack parametrizing flat objects in $\aaa$, is algebraic and locally of finite type over $k$. Then any quasi-compact closed substack $\xx \subset \mm_\aaa$ admits a proper good moduli space.
     \end{thm}
    Lately, there has been a great amount of interest surrounding these moduli problems (\cite{herrero2025moduli}) with applications ranging from moduli of perverse sheaves \cite{lampetti2025good} to cohomological applications \cite{davison2024purity}, \cite{bu2025cohomology}. In the context of Bridgeland stability, one of the more striking applications of \cite{alper2023existence} is the following existence theorem. Let $X$ be a smooth and projective scheme over an algebraically closed field. A stability condition on $X$, $\sigma=(\aaa,Z)\in\stab(X)$, consist of the heart of a $t$-structure $\aaa\subset D^b(X)$ and a \emph{central charge} homomorphism $Z:K_0(D^b(X))\ra\C$ together with some compatibility properties. Assume that $\sigma$ is algebraic, meaning that $\im(Z)\subset\Q[i]$ (hence the heart $\aaa$ is a noetherian abelian category). Assume also that:
    \begin{enumerate}
        \item\label{genericflatness} the heart $\aaa$ satisfies openness of flatness,
        \item $\aaa$ is bounded with respect to the standard $t$-structure on $D^b(X)$,
        \item$\mm_\sigma^{ss}(v)$ is bounded for all $v\in K_0(D^b(X))$
    \end{enumerate}
    then the Harder--Narasimhan filtration gives a $\Theta$-stratification on $\mm_\sigma(v)$ whose open stratum is exactly $\mm_\sigma^{ss}(v)$, hence it admits a proper good moduli space by \cite[Theorem 6.5]{alper2023existence}. This, together with the \emph{deformation argument} by Piyaratne--Toda (\cite{TODA_k3}, \cite[Proposition 4.12]{piyaratne2019moduli}) implies that the same holds for every algebraic stability condition $\tau$ in the same connected component of $\stab(X)$ containing $\sigma$. This gives a bittersweet answer to questions \ref{questionEX} and \ref{questionDEF}: the hypotheses which grant the existence of moduli spaces are well understood, yet there is no general technique available to tackle the problem in full generality. 
    
    The issue of the existence of stability conditions has been solved in the recent breakthrough \cite{li2026remark}, where C. Li manages to prove nonemptiness of $\stab(X)$ for every smooth projective variety $X$. Proving that the newly built stability conditions verify the conditions above would allow for a huge source of new examples of moduli spaces.

    In the relative setting, the quest of constructing stability conditions for smooth and proper families $X\ra S$, allowing for a rich moduli theory, has been addressed in \cite{bayer2021stability}. Notably, it does not suffice to have a collection $\{\sigma_s\}_{s\in S}$ of stability conditions on $D^b(X_s)$ for all $s\in S$ and impose a local constancy property for the central charges. The right requirement turns out to be a global condition on $D^b(X_C)$ for every smooth curve with a morphism $C\ra S$, called a \emph{Harder--Narsimhan structure}. The extension of such formalism to the noncommutative setting has been written recently in \cite{pertusi2026noncommutative}.

    On the other hand, the first extension of the theory of stability conditions in the enhanced setting appeared in the work of Halpern-Leistner--Robotis \cite{halpernleistner2025spaceaugmentedstabilityconditions}. For a smooth and proper category, linear over an excellent base algebra, they define a new structure, called a \emph{locally constant stability condition} (Definition \ref{defstabf}) with very convenient properties for the existence of good moduli spaces. Specifically, in their approach there are two major contributions:
    \begin{itemize}
        \item the notion of a \emph{sluicing} (Definition \ref{defsluicing}) overcomes the noetherianity assumption on the $t$-structures necessary to apply Theorem \ref{thmexistence},
        \item the generic flatness property in \ref{genericflatness} is shown to be equivalent to an explicit bound on the dimension of hom-spaces in the category, called a \emph{mass-hom bound} (see Theorem \ref{thmhlr}).
    \end{itemize}
    We say that a locally constant stability condition is \emph{proper} if, loosely speaking, it satisfies openness of the semistable locus and boundedness of semistable objects. Our main contribution is the following.
    \begin{introthm}[Theorem \ref{thmmain}]
        Let $\dd$ be a smooth and proper linear category over a noetherian algebra $R$ of characteristic zero and finite global dimension. The set $\stab(\dd/R)$ of proper locally constant stability conditions has a structure of a complex manifold for which the forgetful map 
        \begin{equation*}
        \begin{split}
            \stab(\dd/R)&\ra\stab(\dd_\kappa(t))\\
            \sigma&\mapsto\sigma_{\kappa(t)}
        \end{split}
        \end{equation*}
        is a local isomorphism for all $t\in\spec R$.
    \end{introthm}
   The main consequence of this theorem is that if one finds a proper $\sigma\in\stab(\dd/R)$, the same holds for any $\tau$ in the whole connected component of the topological space $\stab(\dd/R)$ containing $\sigma$. This means in particular that having mass-hom bounds for the category $\dd$ depends only on connected components of $\stab(\dd/R)$.

   Our second result is a way to build relative stability conditions from collections of stability conditions on the fibres $\dd_{\kappa(t)}$ for $t\in\spec R$
   \begin{introthm}[Corollary \ref{corcons}]
       Suppose we have a proper stability condition $\sigma_\kappa=(\cal P_\kappa,Z_k)$ on $\dd_\kappa$ for every residue field $\kappa$ of $R$ with locally constant central charges and a $C$-local heart $\aaa_C$ for every DVR $C$ essentially of finite type over $R$ with the generic flatness property. Suppose also that $\aaa_C$ restricts to $\cal P_K(0,1]$ and $\cal P_{C/\mathfrak{m}_C}(0,1]$, where $K$ is the fraction field of $C$. Then there exists a unique proper locally constant stability condition $\sigma\in\stab(\dd/R)$ restricting to $\sigma_K$ and $\sigma_{C/\mathfrak{m}_C}$.
   \end{introthm}
   To the author's knowledge, the literature is missing a comparison between the approach by \cite{halpernleistner2025spaceaugmentedstabilityconditions} and that in \cite{bayer2021stability}. 
   \begin{introthm}[Proposition \ref{propcomparison}]
       The datum on a proper locally constant stability condition on a smooth and proper category $\dd$ is equivalent to the datum of a flat family of stability conditions, as defined in \cite{bayer2021stability}.
   \end{introthm}
   Lastly, given the compactification provided in \cite{halpernleistner2025spaceaugmentedstabilityconditions} by augmented stability conditions, we speculate in \ref{conjaugmented} whether such a construction could be made possible in the relative setting as well, whether a formalism of quasi-convergent paths and its consequences can be carried in this context, and whether properness of locally constant stability conditions is a property which should be expected to hold for all linear smooth and proper categories. 
\subsubsection*{Conventions} 
Although the theory of Bridgeland stability conditions does not have any restrictions coming from the characteristic of the ground field $k$, that of good moduli spaces does. In fact, even the étale local structure of algebraic stacks in \cite{alper2020luna}, necessary for the existence results in \cite{alper2023existence}, works only in characteristic zero. The extension of such a theory in positive characteristic is a deep problem in the field (see \cite{alper2026intrinsicapproachmodulitheory}). To overcome these pathologies, throughout the paper we will assume $\text{char}(k)=0$ for every field $k$. For $t$-structures we use the cohomological notation; all functors are derived unless specified.
\subsubsection*{Acknowledgements}
I would like to thank J. Alper for introducing me to the beautiful theory of good moduli spaces and for his support and guidance. During the preparation of this work, I have had to pleasure to have many insightful conversations with N. Bignami, B. Fantechi, D. Halpern-Leistner, E. Lampetti, E. Macrì, M. Montagnani, L. Morstabilini, E. Pavia, M. Pernice, A. Ricolfi, A. Robotis, P. Stellari, and O. van Garderen, whom I gratefully acknowledge. Lastly, I would like to thank the Department of Mathematics at the University of Washington for the active and warm atmosphere where part of this project took place. The author is part of INDAM-GNSAGA.
\section{Moduli of stable objects in a linear category}
\subsection{Linear categories}
We collect some introductory facts on $S$-linear categories following largely \cite{PERRY2019877} and \cite{hotchkiss2024period}. 
\begin{defi}
    A derived algebraic stack $S$ is \textit{perfect} if
    \begin{itemize}
        \item the diagonal of $S$ is affine,
        \item $D_{qc}(S)$ is compactly generated,
        \item compact and perfect objects in $D_{qc}(S)$ coincide.
    \end{itemize}
\end{defi}
For a perfect derived algebraic stack $S$, the category $\Perf(S)$ may be regarded as a commutative algebra object in the category of small, idempotent-complete, stable $\infty$-categories $\Cat_{st}$. 
\begin{defi}
    An $S$-linear category is a $\Perf(S)$-module object in $\Cat_{st}$.
\end{defi}
The collection of $S$-linear categories has the structure of an $\infty$-category $$\Cat_S:=\Mod_{\Perf(S)}(\Cat_{st}),$$
together with a symmetric monoidal structure denoted $\cal C\otimes_{\Perf(S)}\dd$. An \textit{$S$-linear functor} is a morphism $\cal C\ra \dd$ in $\Cat_S$, and the collection of such forms another $S$-linear category, $\Fun_S(\cal C,\dd)$. To be more precise, given $\cal C,\, \dd,\, \cal F\in\Cat_S$, the $S$-linear functors $\cal C\otimes_{\Perf(S)}\dd\ra\cal F$ classify the bilinear maps $\cal C\times\dd\ra\cal F$. There is a canonical functor $$\cal C\times\dd\ra\cal C\otimes_{\Perf(S)}\dd,$$
and in particular, given objects $C\in\cal C$ and $D\in\dd$, we denote $C\boxtimes D\in\cal C\otimes_{\Perf(S)}\dd$ their image under this functor.
\begin{lm}[\cite{PERRY2019877}, Lemma 2.7]
    Let $\cal C$ and $\dd$ be $S$-linear categories. The category $\cal C\otimes_{\Perf(S)}\dd$ is thickly generated by objects of the form $C\boxtimes D$, with $C\in\cal C$ and $D\in\dd$. 
\end{lm}
Fix a morphism $T\ra S$ of perfect derived algebraic stacks. The product $$\dd_T:=\dd\otimes_{\Perf(S)}\Perf(T)$$
is called the \textit{base change} of $\dd$ along $T\ra S$: it is naturally a $T$-linear category. Moreover, for every $S$-linear functor $\Phi:\cal C\ra \dd$ there is a $T$-linear functor $\Phi_T:\cal C_T\ra \dd_T$ obtained from base-change. For any point $s\in S$, the base change along $\spec(\kappa(s))\ra S$ gives a $\kappa(s)$-linear category $\dd_{\kappa(s)}$, called the \textit{fibre} of $\dd$ over $s$.\\

Let $\dd$ be an $S$-linear category. For objects $E,\,F\in\dd$ there is a \textit{mapping object} $\hhom_S(E,F)\in D_{qc}(S)$ defined by the property 
$$\map_{D_{qc}(S)}(G,\hhom_S(E,F))\cong\map_\dd(E\boxtimes G,F),\hbox{ for }G\in\Perf(S).$$
Specifically, the functor $$\map_\dd(E\otimes(-),F):D_{qc}(S)\ra (\operatorname{Spaces})$$
is representable by \cite[Proposition 5.5.2.2]{lurie2009higher} and $\hhom_S(E,F)\in D_{qc}(S)$ is the representing object by definition.
\subsubsection{Presentable linear categories} 
Often, for applications it is more suited to work with larger categories. In the geometric setting, this usually translates to working with unbounded, quasi-coherent categories. In general, this corresponds to enlarge the category considered to the $\infty$-category $\precat_{st}$ of \textit{presentable stable $\infty$-categories}. Once again, the category $D_{qc}(S)$ is a commutative algebra object in $\precat_{st}$. As above, a \textit{presentable $S$-linear category} is a $D_{qc}(S)$-module object in $\precat_{st}$. The collection of such forms an $\infty$-category $$\precat_S:=\Mod_{D_{qc}(S)}(\precat_{st}).$$
If $\dd\in\Cat_S$ there is a corresponding category $\Ind(\dd)\in\precat_S$ called the \textit{ind-completion} of $\dd$, obtained roughly as freely adjoining all filtered colimits to $\dd$. This correspondence is functorial and induces $$\Ind:\Cat_S\ra\precat_S.$$
Moreover, the functor $\Ind(-)$ factors through an equivalence to the category $\precat_S^\omega$ of compactly generated, presentable, $S$-linear categories. The inverse equivalence $(-)^c:\precat_S^\omega\ra\Cat_S$ is given by taking compact objects. 
\subsection{Smooth and proper categories} 
\begin{defi}
    Let $\dd$ be an $S$-linear category.
    \begin{itemize}
        \item $\dd$ is \textit{proper} if $\hhom_S(E,F)\in\Perf(S)$ for all $E,\,F\in\dd$,
        \item  $\dd$ is \textit{smooth} if $\id_{\Ind(\dd)}\in\Fun_S(\Ind(\dd),\Ind(\dd))$ is representable by a compact object.
    \end{itemize}
\end{defi}
These notions are stable under base change, as shown below.
\begin{lm}[\cite{PERRY2019877}, Lemma 4.10]
    Let $\dd$ be an $S$-linear category, and let $T\ra S$ be a morphism of perfect derived algebraic stacks.
    \begin{itemize}
        \item If $\dd$ is smooth over $S$ then $\dd_T$ is smooth over $T$,
        \item if $\dd$ is proper over $S$ then $\dd_T$ is proper over $T$.
    \end{itemize}
\end{lm}
There is a more intrinsic characterization of the property of being smooth and proper.
\begin{defi}
    Let $(\cal C,\otimes,1_{\cal C})$ be a symmetric monoidal $\infty$-category. An object $A\in\cal C$ is \textit{dualizable} if there is an object $A^\vee\in\cal C$ and morphisms $$\operatorname{coev}_A:1_{\cal C}\ra A\otimes A^\vee\quad\hbox{and}\quad\operatorname{ev}_A:A^\vee\otimes A\ra 1_{\cal C}$$
    such that the compositions 
    \begin{align*}
        A\xrightarrow{\operatorname{coev}_A\otimes\id_A}&A\otimes A^\vee\otimes A\xrightarrow{\id_A\otimes\operatorname{ev}_A} A\\
        A^\vee\xrightarrow{\id_A\otimes\operatorname{coev}_A}&A^\vee\otimes A\otimes A^\vee\xrightarrow{\operatorname{ev}_A\otimes\id_{A^\vee}}A^\vee
    \end{align*}
    are equivalent to the identity morphisms of $A$ and $A^\vee$ respectively.
\end{defi}
\begin{prop}[\cite{PERRY2019877}, Lemma 4.8]
    Let $\dd$ be an $S$-linear category. Then $\dd$ is smooth and proper over $S$ if and only if it is dualizable as an object of $\Cat_S$. In this case the dual is given by the opposite category $\dd^\vee=\dd^{op}$.
\end{prop}
Lastly, we report some features of $S$-linear categories, which will be extensively used throughout this paper.
\begin{lm}[\cite{PERRY2019877}, Lemma 2.9]\label{lmperryprod}
    Let $\cal C$ and $\dd$ be $S$-linear categories. If $C_1,\, C_2\in\cal C$ and $D_1,\,D_2\in\dd$, then the $D_{qc}(S)$-vaued mapping object between $C_1\boxtimes D_1$ and $C_2\boxtimes D_2$ in $\cal C\otimes_{\Perf(S)}\dd$ satisfies $$\hhom_S(C_1\boxtimes D_1, C_2\boxtimes D_2)\cong\hhom_S(C_1,C_2)\otimes\hhom_S(D_1,D_2)$$
\end{lm}
\begin{lm}[Projection formula]\label{lmprojformula}
    Let $S=\spec R$ for a noetherian ring $R$ of finite global dimension. Let $E,\, F\in\dd$ and $I$ be an $R$-module. There is a projection formula of the form
    $$\hhom_S(E,F\boxtimes I)\cong \hhom_R(E,F)\otimes I.$$
\end{lm}
\begin{proof}
    Use Lemma \ref{lmperryprod} to get 
    \begin{align*}
        \hhom_R(E,F\boxtimes I)&\cong\hhom_R(E\boxtimes R,Q\boxtimes I)\\
        &\cong\hhom_R(E,F)\otimes\hhom_R(R,I)\\
        &\cong\hhom_R(E,F)\otimes I.
    \end{align*}
    This concludes the proof.
\end{proof}
\subsection{Pseudo-coherent and bounded coherent categories}
Let $\cal C$ be an $S$-linear category, and suppose that $\tau=(\cal C^{\leq 0},\cal C^{\geq  0})$ is a nondegenerate $t$-structure on $\cal C$. 
\begin{ex}
    This is the case when $\dd$ comes with a compact generator $G\in \cal C^\omega$ - indeed the category $\cal C_G^{\leq 0}$, defined to be the smallest subcategory containing $G$ which is closed under extensions and small colimits, is the aisle of  $t$-structure by \cite[Propoposition 1.4.4]{ha}.
\end{ex}
There are two interesting categories attached to this datum. 
\begin{defi}
    Let $\cal C$ be a presentable, $S$-linear category with arbitrary coproducts, and let $\tau=(\cal C^{\leq 0},\cal C^{\geq 0})$ be a nondegenerate $t$-strcture on $\cal C$.
    \begin{itemize}
        \item The \textit{pseudo-coherent category} $\cal C_{pc}$ consists of objects $E\in \cal C$ such that for all $m\in\Z_{\geq 0}$ there is an exact triangle $$E_m\ra E\ra T_m\xrightarrow{[1]},$$
        with $E_m\in\cal C^\omega$ and $T_m\in\cal C^{\leq-m}$,
        \item The \textit{bounded coherent} category $\cal C_{Coh}^b$ consists of those objects $E\in\cal C_{pc}$ which are bounded with respect to $\tau$. 
    \end{itemize}
\end{defi}
\begin{rmk}
    In the majority of applications, there is a more intrinsic construction for these categories. If the base $S$ is a quasi-compact, quasi-separated scheme then there are equivalences 
    $$\cal C_{pc}=\Fun(\cal C, D_{pc}(S))\quad\hbox{and}\quad\cal C^b=\Fun(\cal C,D^b(S)).$$
\end{rmk}
\begin{ass}\label{assexistence}
    If $S=\spec R$ is an affine scheme and $\dd$ is smooth and proper over $R$, we will assume for the rest of the paper that the category $\Ind(\dd)_{pc}$ satisfies \textit{Grothendieck existence theorem}. Namely, let $(A_n)$ be an inverse system of $R$-algebras with surjective transition maps whose kernels are nilpotent. We will assume that there is an equivalence $$\Ind(\dd)_{pc}\otimes_R A\cong\lim\limits_{\longleftarrow}\left(\Ind(\dd)_{pc}\otimes_R A_n\right).$$
\end{ass}
\subsection{Base change of Ind-Noetherian $t$-structures} 
An essential tool for studying moduli stacks of stable objects in linear categories has been the formalism of base change of $t$-structures. The topic was first addressed in \cite{tarrio2003construction}, \cite{AbramovichPolishchuk+2006+89+130} and \cite{polishchuk2007constant}. Later, it has been refined in \cite{halpern2014structure}, \cite{bayer2021stability}, and \cite{halpernleistner2025spaceaugmentedstabilityconditions}. Here, we present a short account based on the latter article. 
\begin{defi}
    An accessible $t$-structure on a compactly generated, presentable, stable $\infty$-category $\cal C$ is called \textit{Ind-noetherian} is it preserves the subcategory of compact objects $\cal C^\omega$ and the induced $t$-structure on $\cal C^\omega$ is bounded and noetherian.
\end{defi}
\begin{defi}
   Let $\dd$ be an $S$-linear category, and consider a $t$-structure $\tau=(\tau^{\leq 0},\tau^{\geq 0})$ on $\dd$. We say that $\tau$ is \emph{local} over $S$ if for each open embedding $U\xrightarrow{j}S$ there exists a $t$-structure $\tau_U$ on the base change $\dd_U$ such that the functor $j^*:\dd_U\ra\dd$ is exact with respect to $\tau$ and $\tau_U$.
\end{defi}
\begin{rmk}
    Suppose $\tau$ is local over $S$ and Ind-noetherian. Then for each $U\hookrightarrow S$ as above the induced $t$-structure $\tau_U$ is Ind-noetherian.
\end{rmk}
Fix a field $k$, and let $R$ be a noetherian $k$-algebra. Let $\dd$ be a smooth and proper $R$-linear category with a bounded, noetherian, nondegenerate $t$-structure $(\dd^{\leq0},\dd^{\geq0})$. For an $R$-algebra $T$ one has a induced $t$-structure on the base change $\Ind(\dd_T)$, where $\Ind(\dd_T)^{\leq 0}$ is the smallest full, saturated subcategory which is closed under small colimits and extensions and contains the image of the aisle $\dd^{\leq 0}$ under the base change functor $T\otimes_R(-):\dd\ra\Ind(\dd_T)$ \cite[Propositions 1.4.4.11, 1.4.4.13]{ha}.
\begin{warning}
    Throughout the paper we will write $\dd_T$ instead of $\dd_{\spec T}$ to indicate the base-changed category along a ring map $R\rightarrow T$.
\end{warning}
\begin{prop}[\cite{halpernleistner2025spaceaugmentedstabilityconditions},  Proposition 2.17]\label{lmbasechange}
    Suppose that $R \rightarrow T$ is a composition of ring homomorphisms of the following types: 
\begin{enumerate}   
    \item a polynomial algebra, 
    \item a ring localization, or 
    \item $T$ is perfect as an $R$-module.
\end{enumerate}    
Then the base-changed $t$-structures $(\Ind(\dd_T)^{\leq0},\Ind(\dd_T)^{\geq 0})$ are ind-noetherian.
\end{prop}
\begin{rmk}
    With these hypotheses on $R$, the statement of Proposition \ref{lmbasechange} holds for any $R$-algebra $T$ essentially of finite type. Indeed, in this case $R\ra T$ can be factored as a composition of ring homomorphisms of the type appearing in Proposition \ref{lmbasechange}. See \cite[Example 2.18]{halpernleistner2025spaceaugmentedstabilityconditions}.
\end{rmk}
\begin{defi}\label{defloc}
    A functor $f^*:\cal B\ra\cal C $ between compactly generated, $R$-linear categories is called a \textit{localization} if $f^*$ admits a right adjoint $f_*$ which commutes with filtered colimits and such that the counit of the adjunction $f^*f_*\ra\id_\cal C$ is an isomorphism of functors.
\end{defi}
The behaviour of localizations is generally well understood by works as \cite{thomason2013higher} and \cite{neeman1996grothendieck}. We report a corresponding result in our setting.
\begin{lm}[\cite{halpernleistner2025spaceaugmentedstabilityconditions}, Lemma 2.20]\label{lmloc}
Let $f^*:\dd'\rightarrow\dd$ be a localization of compactly generated $R$-linear categories. Suppose that $\dd$ admits an Ind-noetherian $t$-structure $(\dd^{\leq0},\dd^{\geq 0})$ such that $f_*\circ f^*$ is $t$-exact. Then the $t$-structure induced on $\dd'$ by setting $(\dd')^{\leq 0}$ to be generated by $f^*(\dd^{\leq 0})$ under colimits is also ind-noetherian. Moreover the restriction $f^*:\dd^\omega\rightarrow(\dd')^\omega$ is essentially surjective and $t$-exact.
\end{lm}
\begin{ex}
    Let $\dd$ be an $R$-linear category. If $f:R\ra T$ is a ring localization, then the induced functor $f^*:\dd\ra\dd_T$ is an example of a localization in the sense of Definition \ref{defloc}. Indeed $R\ra T$ is exact, so that $f_*(f^*(-))\cong T\otimes_R(-)$ is $t$-exact. In particular, Lemma \ref{lmloc} applies.
\end{ex}
\begin{cor}\label{coropen}
Let $U=\spec(T)\xrightarrow{j}\spec(R)$ be the inclusion of an open subset, and let $f:E\rightarrow F$ be a morphism in $\dd_U$. Then there exist objects $\hat{E},\,\hat{F}\in\dd$ and a morphism $\hat{f}:\hat{E}\rightarrow\hat{F}$ such that $j^*\hat{f}=f$. If $E,\,F\in\dd_U^\heartsuit$ then $\hat{E},\,\hat{F}\in\dd^\heartsuit$ the lift $\hat{f}$ is injective or surjective if $f$ is.
\end{cor}
\begin{proof}
Let $\hat{E},\,\hat{F}\in\dd$ be some lifts to $\dd$ of $E$ and $F$, which exist by Lemma \ref{lmloc}. Then by adjunction it follows that 
\begin{align*}
    \hhom_T(E,F)&\cong\hhom_R(\hat{E},j_*j^*F)\\
    &\cong\hhom_R(\hat{E},\hat{F}\otimes_R T).
\end{align*}
By Lemma \ref{lmprojformula} we conclude that the latter is isomorphic to $\hhom_R(\hat{E},\hat{F})\otimes T$. Thus, up to replacing $\hat{F}$ by $\hat{F}\otimes_R T$ we get a lift $\hat{f}$. For the second statement, by $t$-exactness, $\hat{f}$ will be surjective or injective if $f$ is.
\end{proof}
\begin{lm}\label{lmt-gluing}
    Let $\cal C_i$ be a diagram of presentable, compactly generated, stable $\infty$-categories and functors between them indexed by a category $I$. Assume that each $\cal C_i$ has a $t$-structure $(\cal C_i^{\leq0},\cal C_i^{\geq 0})$ and that each functor between the $\cal C_i$'s is $t$-exact with respect to these $t$-structures. Then $\cal C:=\lim\limits_{\longleftarrow } C_i$ admits a $t$-structure $\left(\lim\limits_{\longleftarrow}\cal C_i^{\leq 0},\lim\limits_{\longleftarrow}\cal C_i^{\geq 0}\right)$.
\end{lm}
\begin{proof}
    This is nothing but the translation of \cite[Lemma 2.2.7]{halpern2020derived} in the setting of $t$-structures. For $i\in I$, define $\tilde{\cal C}_i\subset\Fun(\Delta^1\times \Delta^1,\cal C_i)$ to be the sub $\infty$-category whose objects are cartesian diagrams of the form 
    \begin{equation}\label{eqsquare}
    \begin{tikzcd}
        C^{\leq 0}\arrow[r]\arrow[d]& C\arrow[d]\\
        0\arrow[r]&C^{\geq 1},
    \end{tikzcd}
    \end{equation}
    with $C^{\leq 0}\in \cal C_i^{\leq 0}$ and $C^{\geq 1}\in\cal C_i^{\geq 1}$. Since we assumed that $\cal C_i\ra\cal C_j$ is $t$-exact, there is an induced functor $\tilde{\cal C}_i\ra\tilde{\cal C}_j$. Lastly, for the fact that $(\cal C_i^{\leq 0},\cal C_i^{\geq 1)}$ is a $t$-structure, for each $i$ the functor $\tilde{\cal C}_i\ra\cal C_i$ taking the square \eqref{eqsquare} to $C$ is an equivalence. As a result, we have a diagram of $\infty$-categories sending $i\mapsto\tilde{\cal C}_i$ and a map of diagrams $\{\tilde{\cal C}_i\}_{i\in I}\ra\{\cal C_i\}_{i\in I}$. In particular $$\lim_{\longleftarrow}\tilde{\cal C}_i\xrightarrow{\cong}\lim_{\longleftarrow}\cal C_i.$$
    In a similar fashion, the categories $\cal C_i^{\leq 0}$ and $\cal C_i^{\geq 1}$ may be replaced with equivalent subcategories $\tilde{\cal C}_i^{\leq0}$ and $\tilde{\cal C}_i^{\geq 1}$ of $\Fun(\Delta^1,\cal C_i)$, consisting of diagrams of the form $0\ra C^{\leq 0}$ and $C^{\geq 1}\ra0$ respectively. Thus, there is a map of diagrams $$\begin{tikzcd}
        \{\tilde{\cal C}_i^{\leq 0}\}\arrow[r, bend left=35
        , "\iota_{\leq0}"]&\arrow[l,bend left=35,"\pi_{\leq0}"]\{\tilde{\cal C}_i\}\arrow[r, bend left=35,"\pi_{\geq 1}"]&\arrow[l, bend left=35,"\iota_{\geq 1}"]\{\tilde{\cal C}_i^{\geq1}\},
    \end{tikzcd}$$
    where:
    \begin{itemize}
        \item the functors $\iota_{\leq 0}$, $\iota_{\geq 1}$ map a square of the form \eqref{eqsquare} to $0\ra C^{\leq 0}\in\tilde{\cal C}_i^{\geq 0}$ and $C^{\geq 1}\ra 0\in \tilde{\cal C}_i^{\geq1}$ respectively,
        \item objects $0\ra C\in\tilde{\cal C}_i^{\geq 0}$ and $C\ra 0\in \tilde{\cal C}_i^{\geq1}$ are mapped respectively to the squares $$
        \begin{tikzcd}
            C\arrow[d]\arrow[r,"\id"]&C\arrow[d]&0\arrow[d]\arrow[r]&C\arrow[d,"\id"]\\
            0\arrow[r]&0&0\arrow[r]&C
        \end{tikzcd}$$
    \end{itemize}
    As in \cite{halpern2020derived}, for each $i\in I$ there are natural transformations $\eta_{\geq 1}:\iota_{\geq 1}\circ\pi_{\geq 1}\ra\id_{\tilde{\cal C}_i}$ and $\epsilon_{\geq 1}:\id_{\tilde{\cal C}_i^{\geq 1}}\ra\pi_{\geq 1}\circ\iota_{\geq 1}$ which correspond to the unit and counit of an adjunction $\iota_{\geq1}\dashv\pi_{\geq 1}$. Passing to the limits, the adjunction is preserved and gives a pair of adjoint functors $\left(\lim\limits_{\longleftarrow}\iota_{\geq1}\right)\dashv\left(\lim\limits_{\longleftarrow}\pi_{\geq 1}\right)$. A similar argument for $\iota_{\leq 0}$ and $\pi_{\leq 0}$ allows to build an adjuction $\left(\lim\limits_{\longleftarrow}\pi_{\leq0}\right)\dashv\left(\lim\limits_{\longleftarrow}\iota_{\leq0}\right)$. Now we claim that $$(\tilde{\cal C}^{\leq 0},\tilde{\cal C}^{\geq 1}):=\left(\lim_{\longleftarrow}\tilde{\cal C}_i^{\leq0},\lim_{\longleftarrow}\tilde{\cal C}_i^{\geq 1}\right)$$
    defines a $t$-structure. Orthogonality of the two categories follows immediately from the formula 
    $$
    \map_{\lim\limits_{\longleftarrow}\tilde{\cal C}_i}(\{F_i\},\{G_i\})\cong\lim\limits_{\longleftarrow}\map_{\tilde{\cal C}_i}(F_i,G_i).
    $$
    Invariance under (positive or negative) shifts is clear, and the adjunction described above gives a description of $\pi_{\leq 0}$ and $\pi_{\geq 1}$ as the corresponding truncation functors.
\end{proof}
\begin{prop}\label{propstack}
Let $(\dd^{\leq 0},\dd^{\geq 0})$ be an $R$-local $t$-structure, and let $Y$ be an algebraic stack over $R$. There is a canonical $t$-structure on $\Ind(\dd_Y)$ such that 
\begin{equation}\label{eqbct-str}
\Ind(\dd_Y)^{\leq 0}:=\{E\in \Ind(\dd_Y)\,|\,j^*E\in(\Ind(\dd_T))^{\leq 0}\hbox{ for all smooth }\spec(T)\xrightarrow{j}S\}.
\end{equation}
Furthermore, if $Y$ admits a cover by a disjoint union of $R$-algebras essentially of finite type, each of finite global dimension, the truncation functors preserve the subcategory of compact objects in $\Ind(\dd_Y)$ and the induced $t$-structure on $\Ind(\dd_Y)$ is Ind-noetherian.
\end{prop}
\begin{proof}
The proof follows very closely \cite[Corollary 6.1.3]{halpern2014structure}. Take a smooth hypercover $Y_\bullet\ra Y$ by a simplicial scheme, where each $Y_n$ is disjoint union of affine schemes. On each $\spec T\hookrightarrow Y_n$ Proposition \ref{lmbasechange} gives an Ind-noetherian $t$-structure on $\Ind(\dd_T)$. By hyperdescent, we have an equivalence of categories $$\Ind(\dd_Y)\cong\lim_{\longleftarrow}\Ind(\dd_{Y_\bullet})$$
between the Ind-completion of $\dd_Y$ and the totalization of the diagram of stable $\infty$-categories $\dd_{Y_\bullet}$. By flatness of the cover, all the functors in the diagram are $t$-exact for the pulled-back $t$-structure at each level. We conclude by Lemma \ref{lmt-gluing} the existence of a $t$-structure on $\Ind(\dd_Y)$ described as in \eqref{eqbct-str}. For the second statement, since at each $\Ind(\dd_T)$ the base-changed $t$-structure preserves $\dd_T$ by Proposition \ref{lmbasechange}, to conclude it is sufficient to observe that 
\begin{equation}\label{eqlim}
\Ind(\dd_Y)^\omega=\lim\limits_{\longleftarrow}\Ind(\dd_{Y_\bullet})^\omega=\lim\limits_{\longleftarrow}\dd_{Y_\bullet}.   
\end{equation}
Indeed, since $\dd$ is smooth and proper over $R$, the compact objects in $\Ind(\dd)$ are the dualizable ones, for which \eqref{eqlim} holds.
\end{proof}
\subsection{Stability conditions} The notion of stability conditions was developed by Bridgeland in his foundational paper \cite{bridgeland2007stability} in the triangulated context. We briefly recall his constructions.
\begin{defi}
    Let $\cal T$ be a triangulated category. A \textit{slicing} $\cal P$ of $\cal T$ is a collection of full additive subcategories which satisfy the following:
 \begin{itemize}
 \item $\mathcal{P} (\phi +1)=\mathcal{P} (\phi)[1].$ 
\item  If $\phi _1 >\phi _2$ and $A_i \in \mathcal{P} (\phi _i)$, then 
$\Hom_\mathcal{D} (A_1, A_2)=0$. 
\item For a non-zero object $E\in \mathcal{D}$, we have the 
following collection of triangles:
$$
\begin{tikzcd}
0=E_0 \arrow[rr]  & &E_1 \arrow[dl] \arrow[rr] & & E_2 \arrow[r]\arrow[dl] & \cdots E_{n-1}\arrow[rr] & & E_n =E \arrow[dl]\\
&  A_1 \arrow[ul,dashed] & & A_2 \arrow[ul,dashed]& & & A_n \arrow[ul,dashed]&
\end{tikzcd}
$$
such that $A_j \in \mathcal{P} (\phi _j)$ with $\phi _1 > \phi _2 > \cdots >\phi _n$. 
\end{itemize}   
\end{defi}
\begin{defi}
Let $\Lambda$ be a finitely generated, free abelian group, and let $v:K_0(\cal T)\ra \Lambda$ be a group homomorphism.
\begin{enumerate}
    \item A \textit{pre-stability} condition on $\cal T$ is a pair $\sigma=(Z,\cal P)$, where $\cal P$ is a slicing of $\cal T$ and $Z:\Lambda\ra\C$ is a group homomorphism such that for all $0\neq E\in\cal P(\phi)$ we have $Z(v(E))\in\R_{>0}\cdot e^{i\pi\phi}$. The nonzero objects in $\cal P(\phi)$ are called $\sigma$-semistable of phase $\phi$.
    \item A pre-stability condition is a \textit{stability condition} if there exists a quadratic form $Q$ on $\Lambda_\R$ such that 
    \begin{itemize}
        \item $Q$ is negative definite on the kernel of $Z$,
        \item if $E\in\cal T$ is semistable then $Q(v(E))\geq 0$.
    \end{itemize}
    The latter condition is called the \textit{support property} for $\sigma$.
\end{enumerate}
\end{defi}
At times, it is convenient to express stability conditions in an equivalent way as follows. Let $\cal T$ be a triangulated category and let $\aaa$ be the heart of a bounded $t$-structure on $\cal T$. 
\begin{defi}
    A group homomorphism $Z:K_0(\aaa)\ra\C$ is called a \emph{stability function} if for every $0\neq E\in\aaa$ $$\Im(Z(E))\geq0\quad\text{and}\quad\Im(Z(E))=0\implies\Re(Z(E))<0.$$
\end{defi}
Given a stability function on $\cal T$ we have a notion of \emph{slope} 
$$\mu(E):=\begin{cases}
    -\frac{\Re(Z(E))}{\Im(Z(E))}&\text{ if } \Im(Z(E))\neq 0 \\
    +\infty&\text{ otherwise},
\end{cases}$$
and that of a \emph{phase} $\phi(E):=\frac{1}{\pi}\arg(Z(E))\in(0,1]$.
\begin{defi}
    A object $0\neq E\in\aaa$ is said to be \emph{semistable} if for all proper subobjects $0\neq A\hookrightarrow E$ we have $\mu(A)\leq\mu(E)$, it is \emph{stable} if $\mu(A)<\mu(E)$.
\end{defi}
A stability function has the \emph{Harder--Narasimhan} property if any $0\neq E\in\aaa$ admits a filtration $$0=E_0\hookrightarrow E_1\hookrightarrow\dots E_{n-1}\hookrightarrow E_n=E$$
whose quotients $A_i=E_i/E_{i-1}$ are semistable and $\mu(A_i)>\mu(A_{i+1})$ for all $i\geq 1$.
\begin{lm}[\cite{bridgeland2007stability}, Proposition 5.3]\label{lmslicingheart}
    Giving a prestability condition on $\mathcal{T}$ is equivalent to giving a heart of a bounded t-structure $\aaa \subset \mathcal{T}$, and a stability function $Z$, such that for a non-zero object $E\in \aaa$ one has $$ Z(E)\in\{ r\cdot\exp (i\pi \phi) \,|\,r>0,\hbox{ and } 0<\phi \leq 1\},$$ and the pair $(Z, \aaa)$ satisfies the Harder--Narasimhan property.
\end{lm}
\begin{rmk}
    Recall that under this correspondence for every $\phi\in\R$ the category $\cal P(\phi,\phi+1]$, defined as the smallest extension-closed subcategory containing $\{\cal P(\psi)\}_{\psi\in(\phi,\phi+1]}$, gives a heart of a bounded $t$-structure on $\cal T$. By convention, one sets $\aaa:=\cal P(0,1]$.
\end{rmk}
For an object $0\neq E\in\cal T$ the real number $$m(E):=\sum_i |Z(\operatorname{HN}^i(E))|\in\R_{\geq0},$$
where the $\operatorname{HN}^i(E)$'s are the HN factors of $E$, is called the \emph{mass} of $E$.
\begin{defi}\label{defHNP}
     Let $(Z,\aaa)$ be a stability function on the heart of a bounded $t$-structure $\aaa$ with the HN property. For an object $E\in\aaa$, we define the Harder--Narasimhan polygon $\HNP(E)\subset\C$ to be the convex hull of $Z(F)$ as $F\hookrightarrow E$ ranges over all subobjects such that $\mu(F)>\mu(E)$. 
\end{defi}
There is a metric on the set of slicings on $\cal T$ \cite[Section 6]{bridgeland2007stability} defined as 
\begin{equation}\label{eqmetric0}
d(\cal P,\cal Q):=\inf_{\phi\in\R}\{\epsilon\in\R_{\geq 0}\,|\,\cal P(\leq\phi)\subset\cal Q(\leq \phi+\epsilon),\,\cal P(>\phi)\subset\cal Q(>\phi-\epsilon)\}.
\end{equation}
This metric induces a topology on the set of stability conditions $\stab(\cal T)$, defined as the coarsest one such that the forgetful maps to the set of slicings and $\Hom(\Lambda,\C)$ are both continuous.
\begin{thm}[\cite{bridgeland2007stability},\cite{bayer2019short}]\label{thmdefo}
    The topological space $\stab(\cal T)$ is a complex manifold, and the forgetful morphism 
    \begin{align*}
        \Phi:\stab(\cal T)&\ra\Hom(\Lambda,\C)\\
        (Z,\cal P)&\mapsto Z
    \end{align*}
    is a local isomorphism. In particular, assume that $\sigma=(Z,\cal P)\in\stab(\cal T)$ stisfies the support property with respect to a quadratic form $Q$. If $U_Z$ denotes the connected component of the set of central charges whose kernel is negative definite with respect to $Q$ containing $Z$, and $V_Z$ is the conneceted component of $\Phi^{-1}(U_Z)$ containing $\sigma$, then $$\Phi|_{V_Z}:V_Z\ra U_Z$$
    is a covering map.
\end{thm}
Let $R$ be a commutative ring, and let $\dd$ be an $R$-linear category.
\begin{defi}[\cite{halpernleistner2025spaceaugmentedstabilityconditions}, Definition 2.13]\label{defsluicing}
    A \textit{sluicing} on $\dd$ of width $w\in(0,1]$ is a collection of $t$-structures $\tau_\phi=(\dd(>\phi),\dd(\leq \phi))$ indexed by the real numbers $\phi\in\R$ such that:
    \begin{itemize}
        \item[(a)] for all $a\in\R$, $\dd(>a)=\bigcup_{\phi>a}\dd(>\phi)$,
        \item[(b)] For all $0<b-a<w$ there exists a bounded noetherian $t$-structure $(\cal A_{>0}^{a,b},\cal A_{\leq 0}^{a,b})$ with $\dd(>a)\subset\cal A_{>0}^{a,b}\subset\dd(>b-1)$,
        \item[(c)] every $E\in\dd$ is contained in $\dd(>a)\cap\dd(\leq b)$ for some $a<b$.
    \end{itemize}
\end{defi}
Given a sluicing on $\dd$, we denote $\cal P((a,b]):=\dd(>a)\cap\dd(\leq b)$ and $\tau_{(a,b]}=\tau_{>a}\circ\tau_{\leq b}\cong\tau_{\leq b}\circ\tau_{>a}$ the projection functors with respect to the $t$-structures. \\
Lastly, for $E\in\dd$ we define 
$$
   \phi^-(E):=\sup\{a\,:\,E\in\dd(>a)\}\quad\hbox{and}\quad\phi^+(E):=\inf\{b\,:\, E\in\dd(\leq b)\}.
$$
\begin{prop}[\cite{halpernleistner2025spaceaugmentedstabilityconditions}, Proposition 2.14]\label{propsluicingstab}
    Fix a homomorphism $v:K_0(\dd)\ra\Lambda$. The datum of a stability condition on $\dd$ is equivalent to that of a sluicing of width 1 and a group homomorphism $Z:\Lambda\ra\C$ such that there exists a constant $C>0$ with the property that for every interval $I$ of width less than 1, and $0\neq E\in\cal P(I)$, $Z(E)\in\R_{>0} \cdot e^{i\pi I}$ and $|Z(E)|>C\cdot||v(E)||$ 
\end{prop}
\begin{rmk}\label{rmksluicing}
    We emphasize the following properties apperaring in the proof of Proposition \ref{propsluicingstab}.
    \begin{itemize}
         \item[(i)] An object $E\in\dd$ is semistable of phase $\psi$ if $\psi=\phi^+(E)=\phi^-(E)$, or equivalently $E\in\dd_{\leq \psi}$ and $\tau_{\leq\phi}(E)\cong 0$ for all $\phi\leq\psi$. 
         \item[(ii)] The inequality $m(E)\geq Z(E)>C\cdot||v(E)||$ is equivalent to the support property.
         \item[(iii)] The process of obtaining a sluicing from a stability condition is not entirely immediate. Since later in the paper we will need to generalize it, we give a short proof. Let $\sigma\in\stab(\dd)$, and let $\{\cal P(\phi)\}_{\phi\in\R}$ be the corresponding slicing. The $t$-structures $(\cal P(>\phi),\cal P(\leq\phi))$ verify item (a) in the definition of a sluicing, and item (c) follows from the existence of Harder--Narasimhan filtrations in the definition of a slicing. We only need to prove item (b). Recall that, if $\{\cal P(\phi)\}_{\phi\in\R}$ and $\{\cal Q(\phi)\}_{\phi\in\R}$ are two slicings on $\dd$, by \cite[Lemma 6.1]{bridgeland2007stability} the metric in \eqref{eqmetric0} can be also characterised as
         \begin{equation}\label{eqmetric}
         d(\cal P,\cal Q):=\inf\{\epsilon\geq 0\,|\,\cal Q(\phi)\subset\cal P([\phi-\epsilon,\phi+\epsilon])\,\forall\phi\in\R\}.
         \end{equation}Recall also that for an \emph{algebraic} stability condition $\tau=(W,\cal Q)$, i.e. one for which $\im(W)\subset\Q[i]$, the $t$-structures $(\cal Q(>\phi),\cal Q(\leq \phi))$ are Noetherian by \cite[Proposition 5.0.1]{AbramovichPolishchuk+2006+89+130}. Moreover, since the lattice $\Lambda$ is assumed to have finite rank, the set of algebraic stability conditions is dense in $\stab(\dd)$. Let $\tau=(W,\cal Q)$ be algebraic and such that $d(\cal P,\cal Q)=\epsilon$. Let $w\in(0,1)$. Then for every $0<b-a<w$ by \eqref{eqmetric} we have 
         $$ \cal Q(>b-1+\epsilon)\subset\cal P(>b-1)\quad\hbox{and}\quad\cal P(>a)\subset\cal Q(>a-\epsilon).$$
         Choose $\epsilon<\frac{a-(b-1)}{2}$ and pick any $\psi\in(b-1+\epsilon,a-\epsilon)$, so that the corresponding $t$-structure is Noetherian. Then the $t$-structure $(\cal Q(>\psi),\cal Q(\leq \psi))$ verifies 
         \begin{equation}\label{eqmetric2}
         \cal P(>b-1)\subset\cal Q(>\psi)\subset\cal P(>a),
         \end{equation}
         so that item (b) follows.

         Observe moreover that if one takes $\epsilon<\frac{1-w}{2}$ then the choice of an algebraic $\tau=(W,\cal Q)$ does \emph{not} depend on $a,\,b\in\R$. 
    \end{itemize}
\end{rmk}
\subsection{Existence of moduli spaces and consequences}
Let $\dd$ be a smooth and proper category over a Noetherian ring $R$ of characteristic zero and finite global dimension. Consider the functor 
\begin{align*}
\mmm:(\dga/R)&\rightarrow(\infty\,\mhyphen\gpd)\\
T&\mapsto(\dd_T)^{\cong},
\end{align*}
where $(-)^{\cong}$ is the $\infty$-groupoid obtained from $\dd_T$ by discarding all the non invertible 1-morphisms.
\begin{thm}[\cite{toen2007moduli}, Theorem 3.6]\label{thmtv}
    The functor $\mmm$ is a derived algebraic stack, locally of finite presentation over $R$.
\end{thm}
The set of points of $\mm$, denoted $|\mm|$, is defined as the disjoint union of $\pi_0(\mm(K))$ over all fields $K$ over $R$, modulo the equivalence relation generated by $[E]\sim [E\otimes_K L]$ for any extension of fields $K\subset L$ over $R$ and $E\in\dd_K$. A subset $V$ of $|\mm|$ is bounded if there exists a scheme $S$ of finite type over $R$ such that  $V$ is contained in the image of $|S|\ra|\mm|$. A subset $U\subset|\mm|$ is open if for any $|S|\ra|\mm|$ the preimage of $U$ in $|S|$ is open. Lastly, a point $E\in|\mm|$ is of finite type if it has a representative defined over a field which is finitely generated as an $R$-algebra.
\begin{defi}\label{defilocconstantmap}
Fix a finitely generated, free abelian group $\Lambda$. A map $v:|\mm|\ra\Lambda$ is \emph{locally constant} if for every morphism $f:S\ra\mm$, with $S$ a connected scheme, the composition $v\circ f:|S|\ra\Lambda$ is constant. Such a map $v$ is said to be \emph{additive} if the following diagram commutes 
$$
    \begin{tikzcd}
    \left| \mm\times_{\spec R}\mm \right| \arrow[d,"v\times v"']\arrow[r,"\oplus"]& \left| \mm \right| \arrow[d,"v"]\\
    \Lambda\times\Lambda\arrow[r,"+"]&\Lambda,
\end{tikzcd}
$$
where $\oplus$ sends $(E_1,E_2)\in\dd_S^2$ to $E_1\oplus E_2\in\dd_S$.
\end{defi}
\begin{defi}[\cite{halpernleistner2025spaceaugmentedstabilityconditions}, Definition 2.24]\label{defstabf}
        Let $\dd$ be a smooth and proper category over $R$, and let $\Lambda$ be a finitely generated, free abelian group. A \textit{locally constant stability condition} on $\mathcal{D}$, is the data of 
        \begin{enumerate}
            \item a sluicing $\mathcal{P}_\kappa$ of width 1 on $\mathcal{D}_\kappa$ for every residue field $\kappa$ of $R$,
            \item an additive, locally constant $v:|\mathcal{M}|\rightarrow \Lambda$,
            \item a group homomorphism $Z:\Lambda\rightarrow\mathbb{C},$
        \end{enumerate}
        together with the following conditions:
        \begin{itemize}
            \item[(a)] for any DVR $C$ whose fraction field $K$ is a residue field of $R$ there is a sluicing of width 1 on $\mathcal{D}_C$ that induces the same sluicings on $\mathcal{D}_K$ and $\mathcal{D}_{C/\mathfrak{m}_C}$ as the ones in (1),
            \item[(b)] for each $u\in\Lambda$ and for all $x\in|\mathcal{M}|$ such that $v(x)=u$ and $\phi^+(x)-\phi^-(x)<1$, $Z(u)\in\mathbb{R}_{>0}\cdot  e^{i\pi(\phi^-(x),\phi^+(x)]}$ and $|Z(u)|>C\cdot||u||$.
        \end{itemize}
        We denote the set of locally constant stability conditions as $\stab(\dd/R)$.
\end{defi}
\begin{rmk}\label{rmksuppprop}
From the above definition, there are a few remarks to be made:
    \begin{itemize}
    \item[(a)] An additive map $v:|\mm|\rightarrow\Lambda$ naturally descends to a group homomorphism $E\mapsto v(E):K_0(\dd_\kappa)\rightarrow\Lambda$ for any field $\kappa$ over $R$ (see also \cite[Lemma 2.23]{halpernleistner2025spaceaugmentedstabilityconditions}).
    \item[(b)] As a consequence of the previous item, a locally constant stability condition on $\dd$ induces stability condition on $\dd_\kappa$ for any field $\kappa$ over $R$ (\cite[Lemma 2.26]{halpernleistner2025spaceaugmentedstabilityconditions}).
    \item[(c)] As mentioned in Remark \ref{rmksluicing}, the support property is equivalent to the inequality $|Z(u)|>C\cdot||u||$. As a result, since the resulting quadratic form is negative definite on the kernel of $Z$, for any $C>0$ there can be only finitely many classes with $|Z(u)|<C$ and $Q(u)\geq 0$.
    \end{itemize}
\end{rmk}

Recall the notion of a good moduli space. 
\begin{defi}[\cite{alper2013good}]
    A quasi-compact morphism $\varphi: \mathcal{X} \rightarrow X$ from an Artin stack to an algebraic space is a \emph{good moduli space} if 
         \begin{enumerate}
         \item The push-forward functor on quasi-coherent sheaves is exact.
         \item The induced morphism on sheaves $\mathcal{O}_X \rightarrow \varphi_* \mathcal{O}_{\mathcal{X}}$ is an isomorphism.   
         \end{enumerate}
\end{defi}
In particular, $\varphi$ is surjective and universally closed, so $X$ has the quotient topology. If $\mathcal{X}$ is locally noetherian, then $\varphi$ is universal for maps to algebraic spaces.

In our context, the most general result for the existence of good moduli spaces is provided by the following theorem.
    \begin{thm}[\cite{halpernleistner2025spaceaugmentedstabilityconditions}, Theorem 2.31]\label{thmhlr}
        Let $\sigma\in\stab(\dd/R)$ be a locally constant stability condition, and let $G\in\dd$ be a classical generator such that $$\sup_{\substack{K\text{ a field of}\\\text{ finite type over }R}}\max(m(G_K),\phi^+(G_K)-\phi^-(G_K))<\infty.$$
        Then the following are equivalent:
        \begin{enumerate}
            \item There exists a function $f:\mathbb{R}_{>0}\rightarrow\mathbb{Z}_{\geq 0}$ such that for each field $k$ of finite type over $R$ and $E\in\mathcal{D}_k$ $$\dim\Hom(G_k,E)\leq f(m(E)),$$
            \item\label{condition2} For any $C>0$ the set of finite type points $E\in|\mathcal{M}|$ such that $m(E)<C$ and $|\phi^\pm(E)|<C$ is open and bounded,
            \item For every $v\in\Lambda$, the subset of $|\mm|$ parametrizing $\sigma$-semistable objects in $\dd_R^\heartsuit$ of numerical class $v$ is open. The corresponding open substack $\mm_\sigma^{ss}(v)\subset\mm$ has affine diagonal, and the underlying classical stack of $\mm_\sigma^{ss}(v)$ admits a proper good moduli space.
        \end{enumerate}
    \end{thm}

\begin{rmk}
    In view of the next results, we will say that a locally constant $\sigma\in\stab(\dd/R)$ is \emph{proper} if it satisfies any of the equivalent condition of Theorem \ref{thmhlr}.
\end{rmk}
The technical heart to the proof of \ref{thmhlr} lies in the next observation
\begin{prop}[\cite{halpernleistner2025spaceaugmentedstabilityconditions}, Proposition 2.37]
    Under the hypotheses of Theorem \ref{thmhlr}, if condition \ref{condition2} holds then: 
    \begin{itemize}
        \item the substack $\mmm_{\sigma}(\phi,\phi+1]\hookrightarrow\mmm$, whose value at $T\in R\alg$ is $\{E\in\dd_T\,|\,\phi^+(E_t),\phi^-(E_t)\in (\phi,\phi+1]\}$, is open, with affine diagonal, algebraic and locally of finite presentation over $R$,
        \item $\mm_\sigma(\phi,\phi+1]$ has a well-ordered $\Theta$-stratification by Harder--Narasimhan filtrations. The subset $\{x\in|\mm|\,|\,m(x)\leq C\}$ is an open, quasi compact union of strata for every $C>0$,
        \item for every $v\in\Lambda$ the substack $\mm_\sigma^{ss}(v)\hookrightarrow\mm$ is open and bounded and admits a proper good moduli space.
    \end{itemize}
\end{prop}
We end the section by summarizing a few other consequences of Theorem \ref{thmhlr}.
\begin{defi}
Let $E\in|\mm|$, and let $T\in R\alg$. We define the following functions
\begin{align*}
\phi_E^{\pm}:T&\rightarrow\R\cup\pm\infty\\
t&\mapsto\phi^\pm(E_t)
\end{align*}
\end{defi}
Let $T\in R\alg$ be of finite type. 
    \begin{lm}\label{lmphase}
If condition \ref{condition2} from theorem (\ref{thmhlr}) holds, the functions $\phi_E^+$ and $\phi_E^-$ are respectively upper and lower semicontinuous, constructible functions on $T$.
\end{lm}
\begin{proof}
Due to \cite[Lemma 2.27]{halpernleistner2025spaceaugmentedstabilityconditions} , more generally the functions $\phi^{\pm}:|\mm|\rightarrow\R\cup\{\pm\infty\}$ are respectively monotone increasing and decreasing upon specialization, so $\phi_E^\pm$ a fortiori are. It remains to prove that they are constructible functions on $T$. As in \cite[Lemma 20.9]{bayer2021stability}, we may assume that $T$ is irreducible and prove that $\phi_E^\pm$ are constant on an open $U\subset T$. Let $\xi\in T$ be the generic point, and let $$0\to E_1\to\dots\to E_n=E_\xi$$
be the Harder-Narasimhan filtration of $E_\xi$. By corollary (\ref{coropen}) it can be lifted to a chain of maps $$0\to \hat{E}_1\to\dots\to \hat{E}_n=E$$ in $\dd$. The pullback $(F_{i+1})_\xi=(\operatorname{cofib}(\hat{E_i}\to \hat{E}_{i+1}))_\xi$, being a HN factor for $E_\xi$, is semistable for every $i=\,\dots, n$. By theorem (\ref{thmhlr}(3)) there exists a nonempty open $U\subset T$ with the property that $(F_i)_t$ is semistable for all $t\in U$. As a result, $\phi_E^\pm$ are constant on $U$.
\end{proof}
\begin{cor}\label{corinterval}
    Keep the hypotheses as above. For every interval $I\subset\R$ of length less thank 1 the substack $\mmm_{\sigma}(I)\hookrightarrow\mmm$, whose value at $T\in R\alg$ is $\{E\in\dd_T\,|\,\phi^+(E_t),\phi^-(E_t)\in I\}$, is open, thus algebraic and locally of finite presentation. Moreover $\mm_\sigma(I)(v)$ is bounded for every $v\in \Lambda$.
\end{cor}
\begin{proof}
    It suffices to observe that for $I=(a,b]$ (the argument for other kinds of intervals being completely analogous), $T\in R\alg$ and $E\in\dd_T$ the set $$\{t\in\spec T\,|\,\phi^\pm(E_t)\in I\}=(\phi_E^-)^{-1}(a,+\infty)\cap(\phi_E^+)^{-1}(-\infty,b]$$
    is open by Lemma \ref{lmphase}. Fix $v\in\Lambda$. Up to taking its closure we may assume that $I=[a,b]$. Let $E\in\mm_\sigma(I)(v)(\spec k(t))$ for $t\in\spec R$. Then the central charge of every Harder--Narasimhan factor of $E$ lies in the parallelogram in the complex plane with endpoints 0 and $Z(v)$ and angles $\pi a$ and $\pi b$. By the support property, there are finitely many possible classes $v'\in\Lambda$ arising as semistable factors for $[E_t]$ as $t\in\spec R$ ranges over all points of finite type. Hence, if $\sigma$ is proper then $\mm_\sigma(I)(v)$ is bounded. 
\end{proof}
\section{Local deformations of stability}
\subsection{Flat and torsion objects, representable hom spaces} 
    Let $\dd$ be a smooth and proper category over a noetherian ring $R$ of finite global dimension. Consider a $t$-structure $\tau$ on $\Ind(\dd)$ with heart $\Ind(\dd)^\heartsuit$.
\subsubsection*{Flat objects}
\begin{defi}
For an $R$-algebra $T$, an object $E\in\Ind{\dd_T}$ is \textit{$T$-flat} if $E\otimes_T M\in\Ind(\dd_T)^\heartsuit$ for all $M\in T\mhyphen\Mod$.
\end{defi}
\begin{lm}[\cite{halpern2014structure}, Proposition 6.2.2]\label{lminstab}
For $E \in \Ind(\dd_T)$, the following are equivalent:
\begin{enumerate}
\item $E$ is $T$-flat, i.e. $E \otimes^L_T M \in \Ind(\dd_T)^\heartsuit$, for all $M \in T\mhyphen\Mod$,
\item $g^*(E) \in \dd_S^\heartsuit$ for any $g^* : \dd_T \to \dd_S$ induced by a map of algebras $g:T \to S$, 
\item $E \otimes^L_T (T/I) \in \Ind(\dd_T)^\heartsuit$ for all finitely generated ideals $I \subset T$.
\end{enumerate}
If the $t$-structure on $\dd$ is non-degenerate, these are equivalent to
\begin{enumerate}
\item[(a)] $E \in \Ind(\dd_T)^\heartsuit$ and the functor $E \otimes_T (-) = H^0(E \otimes^L_T(-)) : T\mhyphen\Mod \to \Ind(\dd_T)^\heartsuit$ is exact.
\end{enumerate}
Furthermore, if $T$ is Noetherian and $E \in \dd_{qc,T}^{\leq 0}$ is pseudo-coherent, then these are equivalent to
\begin{enumerate}
\item[(b)] $E|_{T/\mathfrak{m}} \in (\Ind\dd_{T/\mathfrak{m}})^\heartsuit$ for all maximal ideals $\mathfrak{m} \subset T$.
\end{enumerate}
\end{lm}
\begin{proof}
The proof is just the translation of \cite[Lemma 6.2.2]{halpern2014structure} in the setting of an $R$-linear category. The implications $(1)\Rightarrow(2)\Rightarrow (3)$ hold trivially, as $g^*(E)\in\Ind(\dd_S)$ lies in the heart $\Ind(\dd_S)^\heartsuit$ if and only if $g_*g^*(E)=E\otimes_T S\in\Ind(\dd_T)^\heartsuit$.

\emph{$[(3)\Rightarrow(1)]$.} Since $\Ind(\dd_T)^\heartsuit$ is closed under (filtered) colimits and $T\mhyphen\Mod$ is compactly generated by finitely presented modules, it suffices to prove that (1) holds for finitely presented modules. Let $M=T^{\oplus n}/L\in T\mhyphen\Mod$ be finitely presented. By \cite[\href{https://stacks.math.columbia.edu/tag/00HD}{Tag 00HD}]{stacks-project} $M$ admits a finite filtration whose graded pieces are given by $T/I$, with $I\subset T$ an ideal. Now, $\Ind(\dd_T)^\heartsuit$ is closed under extensions so the claim follows if it holds for $M=T/I\in T\mhyphen\Mod$. To conclude, $T/I=\operatorname{colim}_\alpha T/I_\alpha$ where $\{I_\alpha\}_\alpha$ is a filtered system of finitely generated submodules  $I_\alpha\subset I$, and since $\Ind(\dd_T)^\heartsuit$ is closed under filtered colimits the claim follows.

\emph{$[(1)\Rightarrow(4)]$.} The functor $E\otimes_T(-):D_{qc}(T\mhyphen\Mod)\ra\dd_T$ is always left $t$-exact by construction. To prove that it is right $t$-exact it enough to observe that any complex of $T$-modules in cohomological degree $\leq 0$ admits a presentation by a complex $M^\bullet$ of free modules. Now, the product $E\otimes_T M^\bullet$ lies in in the category generated by $E$ under extensions, left shifts, and filtered colimits.

\emph{$[(4)\Rightarrow(1)]$.} As before, for $M\in T\mhyphen\Mod$ consider a presentation $0\ra K\ra T^{\oplus J}\ra M\ra 0$. By hypothesis, $H^0(E\otimes_T(-))$ is exact so the long exact cohomology sequence implies that $H^{-1}(E\otimes_T M)=0$ and $H^{-i}(E\otimes_T K)\cong H^{-i-1}(E\otimes_T M)$ for all $i>0$. This has to hold for all $T$-modules, so $H^{-i}(E\otimes_T M)=0$ for every $i>0$. The $t$-structure is non-degenerate, so $E\otimes_T M\in\Ind(\dd_T)^\heartsuit$.

\emph{$[(3)\Rightarrow(5)]$.} If $T$ is noetherian this is a tautology.

\emph{$[(5)\Rightarrow(3)]$.} Lemma \ref{lmt-gluing} implies that if $S\ra \spec T$ is a faithfully flat map of schemes and $E\in\Ind(\dd_S)$ is $S$-flat then $E$ is $T$-flat. Let $\mathfrak{m}\subset T$ be a maximal ideal, and denote $\hat{T}_{\mathfrak{m}}$ the completion of $T$ at $\mathfrak{m}$. The ring $T$ is noetherian by assumption, so $S:=\bigsqcup_{\mathfrak{m}\subset T}\spec (\hat{T}_{\mathfrak{m}})$ is faithfully flat over $\spec(T)$. The upshot of this discussion is that it is enough to assume that $T$ is a complete, Noetherian local ring with maximal ideal $\mathfrak{m}\subset T$. The pushforward is $t$-exact, so item (3) is equivalent to $E|_{\spec(T/I)}\in\Ind(\dd_{T/I})^\heartsuit$ for all ideals $I\subset T$. For each ideal $I\subset T$, the quotient $R/I$ is still a complete local ring with maximal ideal $\mathfrak{m}/I$. Up to replacing $T$ with $T/I$, it suffices to prove that if $E\in\Ind(\dd)_{ps}^{\leq 0}$ satisfies $E|_{T/\mathfrak{m}}\in\Ind(\dd_{T/\mathfrak{m}})^\heartsuit$ then $E\in\Ind(\dd_T)^\heartsuit$. A simple induction argument for the triangles $$E\otimes_T (\mathfrak{m}^n/\mathfrak{m}^{m+1})\ra E\otimes_T (T/\mathfrak{m}^{n+1})\ra E\otimes_T(T/\mathfrak{m}^n)\xrightarrow{[1]},$$
together with $t$-exactness of pushforward gives $E\otimes_T(T/\mathfrak{m}^n)\in\Ind(\dd_{T})_{ps}^\heartsuit$ for all $n\geq 1$. By Assumption \ref{assexistence}, Grothendieck existence theorem holds for $\Ind(\dd)_{ps}$ so $$E=\operatorname{holim}_n E\otimes_T(T/\mathfrak{m}^n).$$
Homotopy limits are left $t$-exact for accessible $t$-structure on a presentable stable $\infty$-category. In particular, $E\otimes_T(T/\mathfrak{m}^n)\in\Ind(\dd_{T})_{ps}^\heartsuit$ for all $n\geq 1$ implies that $E\in\Ind(\dd_T)^{\geq 0}$ so that $E\in\Ind(\dd_T)^\heartsuit$.
\end{proof}

For what follows we will assume the $t$-structure on $\dd$ to be nondegenerate. Moreover, in accordance with the notation used for stability conditions, we will write $\aaa$ for the heart $\dd^\heartsuit=\tau^{\leq0}\cap\tau^{\geq0}$ and $\aaa_T$ for the base-changed heart along a map of rings $R\ra T$. 

In the case of local rings we have some more characterizations for flat objects, which can be thought as analogues of local flatness criteria.

\begin{lm}\label{lmlocalflatness}
    Let $B\xrightarrow{g} B_0$ be a surjective homomorphism of noetherian local rings with kernel $I$ such that $I^2=0$. An object $E\in\aaa_B$ is $B$-flat if and only if 
    \begin{enumerate}
        \item The object $E_0:=E\otimes_B B_0$ is $B_0$-flat, and
        \item the morphism $E_0\otimes_{B_0}I\rightarrow E$ is injective.
    \end{enumerate}
\end{lm}
\begin{proof}
    Let $E$ be flat over $B$. Observe that by Lemma \ref{lminstab}(2) the object $E_0$ belongs to $\aaa_{B_0}$. Moreover, for a map $h:B_0\ra S$ of algebras we have $h^*(E_0)\cong h^*(g^*(E))\in\aaa_S$ so, again by Lemma \ref{lminstab}(2), $E_0$ is $B_0$-flat. Next, since $E\otimes_B(-)$ is exact by Lemma \ref{lminstab}(a) we have a short exact sequence of the form 
    \begin{equation}\label{eqlocalflatness}
    0\ra E\otimes_B I\ra E\ra E\otimes_B B_0\ra0
    \end{equation}
    in $\aaa_B$. Because $I^2=0$, we have the identity $E\otimes_B I\cong E_0\otimes_{B_0}I$. Form this observation and the exact sequence \eqref{eqlocalflatness} the map $E_0\otimes_{B_0}I\ra E$ is injective.

    Conversely, let $M$ be a $B$-module, we need to prove that $E\otimes_B M\in\aaa_B$. Observe first that the injectivity of $E_0\otimes_{B_0}I\ra E$ implies that $E_0\otimes_{B_0}B\in\aaa_B$. Consider then the short exact sequence $$0\ra IM\ra M\ra M/IM\ra0$$
    and observe that both $IM$ and $M/IM$ are naturally $B_0$-modules. The object $E_0$ is $B_0$-flat and the functor $g_*:\dd_{B_0}\ra \dd_B$ is exact, hence $$g_*(E_0 \otimes_{B_0}^L IM) \in \mathcal{A}_B \quad \text{and} \quad g_*(E_0 \otimes_{B_0}^L (M/IM)) \in \mathcal{A}_B.$$
    Next, consider the induced exact triangle 
    \begin{equation}\label{eqlocalflatness2}
    E \otimes_B IM \rightarrow E \otimes_B M \rightarrow E \otimes_B(M/IM) \xrightarrow{[1]}
    \end{equation}
    and observe that by the previous step and the fact that $E_0\otimes_{B_0}B\in\aaa_B$  we have $$E\otimes_B IM\cong g_*(E_0\otimes_{B_0}IM)\in\aaa_B\quad\text{and}\quad E\otimes_B(M/IM)\cong g_*(E_0\otimes_{B_0}(M/IM))\in\aaa_B.$$
    From this observation, taking the long exact sequence from \eqref{eqlocalflatness2} we have $H_{\aaa_B}^i(E\otimes_B M)=0$ for $i\neq0$, so that $E$ is $B$-flat.
\end{proof}

\begin{lm}\label{lmlocalflatness2}
    Let $C$ be a DVR essentially of finite type over $R$ with uniformizer $\pi$ and residue field $\kappa$. An object in $\aaa_{C/(\pi)^m}$ is flat if and only if multiplication by $\pi$ induces isomorphisms 
    \begin{equation}\label{eqtower}
    E/\pi\cdot E\cong \pi\cdot E/\pi^2\cdot E\cong\dots\cong \pi^{m-1}\cdot E.
    \end{equation}
\end{lm}
\begin{proof}
    Denote $C_m:=C/(\pi)^m$, and denote by $j_*:\dd_{\kappa}\ra\dd_{C_m}$ the pushforward functor. Given $E\in\aaa_{C_m}$, to determine its flatness over $C_m$ it suffices by Lemma \ref{lminstab}(b) and exactness of $j_*$ to show that $E\otimes C_m/\pi\cong j_*j^*E\in\aaa_{C_m}$. Consider the free resolution of $C_m/\pi$ $$\dots\ra C_m\xrightarrow{\pi^{m-1}\cdot\id} C_m\xrightarrow{\pi\cdot\id}C_m\xrightarrow{\pi^{m-1}\cdot\id}C_m\xrightarrow{\pi\cdot\id}C_m,$$
    which is two-periodic. As a result, considering its truncation at the degree -2, and the induced triangle we have $$j_*j^*E[-2]\ra E\xrightarrow{\pi\cdot\id}E\ra j_*j^*E.$$
    The flatness of $E$ is the same as the vanishing $H_{\aaa_{C_m}}^{-1}(j_*j^* E)=H_{\aaa_{C_m}}^{-2}(j_*j^* E)=0$. Now, the condition $H_{\aaa_{C_m}}^{-1}(j_*j^* E)=0$ is equivalent to $\ker(\pi\cdot\id)\cong \pi^{m-1} E$, which in turn is equivalent to the surjectivity of the morphism $E/\pi E\ra\ker(\pi\cdot\id)$. On the other hand, $H_{\aaa_{C_m}}^{-2}(j_*j^* E)=0$ is equivalent to $\ker(\pi^{m-1}\cdot\id)\cong\pi E$. As in the previous step, this is the same as the injectivity of $E/\pi E\hookrightarrow\ker(\pi\cdot\id)$. Thus, the flatness of $E$ is equivalent to the existence of an isomorphism $E/\pi E\cong\ker(\pi\cdot\id)$, which in turn is equivalent to \eqref{eqtower}.
\end{proof}
\begin{defi}
    Let $\sigma$ be a locally constant stability condition on $\dd$. An object $E\in\dd$ is \emph{relatively flat} if for all residue fields $\kappa$ of $R$ the object $E_\kappa\in\dd_\kappa$ is flat with respect to the $t$-structure $\cal P_\kappa(0,1]$ on $\dd_\kappa$.
    A locally constant $\sigma\in\stab(\dd/R)$ has the \emph{generic flatness property} if the set of points in $|\mm|$ whose corresponding objects are flat in $\dd_\kappa$ is open. Equivalently, generic flatness holds if given a family $E\in\mm(T)$  for $T\in R\alg$ there is an open $U\subset T$ such that $E|_U$ is relatively flat.
\end{defi}
\begin{rmk}
    In the context of a locally constant stability conditions, it is clear that the generic flatness property is equivalent to the morphism $\mm_\sigma(0,1]\ra\mm$ being an open immersion.
\end{rmk}
\subsubsection*{Torsion objects and related constructions}
\begin{defi}
An object $E\in\dd$ is called a \textit{torsion object} if it is the pushforward of an object in $\dd_W$ for some proper closed subscheme $W\subset\spec R$. The subcategory of torsion objects is denoted $\dd_{tor}$.  
\end{defi}

\begin{lm}[\cite{bayer2021stability}, Lemma 6.9]\label{lmcurve} 
Let $C$ be a DVR over $R$ with uniformizer $\pi\in C$, and let $W \subset C$ be a 0-dimensional subscheme with ideal sheaf $I_W=(\pi)^n$ for some $n\geq 1$. Let $\aaa_W$ be the heart of the induced t-structure on $\dd_W$. 
Then for any $E \in \aaa_C$ we have short exact sequences
\begin{equation}\label{eqlmcurve}
0 \to \Ann(I_W; E) = i_{W*}^{} H^{-1}_{\aaa_W} \left( E_W \right) \hookrightarrow I_W \otimes E \twoheadrightarrow I_W \cdot E \to 0 \end{equation}
and
\begin{equation*}
0 \to I_W \cdot E \hookrightarrow E \twoheadrightarrow E/I_W\cdot E = i_{W*}^{} H^0_{\aaa_W} \left(E_W \right) \to 0,
\end{equation*} 
where $\Ann(I_W; E)$ and $I_W \cdot E$ denote the kernel and the image of the canonical map $I_W\otimes E\ra E$ induced by the triangle
\begin{equation*}
I_W \otimes E \to E\ra i_{W*}i_W^* E\xrightarrow{[1]} 
\end{equation*}
Moreover, $H^i_{\aaa_W}(E_W) = 0$ for $i \neq -1,0$.
\end{lm}
\begin{proof}
We take cohomology of the exact triangle 
$I_W \otimes E \to E \to i_{W*}i_W^* E$ with respect to $\aaa_C$, using that $(i_W)_*$ is t-exact. 
\end{proof}
We call the essential image of $i_{W*}:\cal A_W\ra\cal A_C$ the category of $W$-torsion objects .
\begin{cor}\label{corsupport}
    Let $W\subset C$ be a 0-dimensional subscheme. Then the subcategory $\aaa_{W,tor}$ of $W$-torsion objects is closed under subobjects, quotients, and extensions.
\end{cor}
\begin{proof}
    By Lemma \ref{lmcurve}, an object $E\in\aaa_C$ is $W$-torsion if and only if $I_W\otimes E\ra E$ is zero.
\end{proof}

\begin{defi}
An object in $\aaa$ is called \textit{torsion-free} if it contains no nonzero torsion subobjects. We denote by $\aaa_{tor}$ and $\aaa_{tf}$ the subcategories of torsion and torsion-free objects respectively.
\end{defi}
\begin{defi}\label{defCtorsion}
    Let $C$ be a DVR essentially of finite type over $R$. The base changed heart $\aaa_C$ is said to have a $C$-torsion theory if the pair of subcategories $(\aaa_{C,tor},\aaa_{C,tf})$ forms a torsion pair.
\end{defi}
\begin{rmk}
If $\Ind (\dd)$ comes with an Ind-noetherian $t$-structure, by Lemma \ref{lmbasechange} the base changed $t$-structure on $\Ind(\dd_T)$ is Ind-noetherian for every $R$-algebra $T$ essentially of finite type. In particular, by noetherianity, objects in the heart $\aaa_T$ have maximal torsion subobjects. Therefore, any object $E\in\aaa_T$ admits a surjective morphism $E\twoheadrightarrow E^{tf}$. In this case the categories $(\aaa_{T,tor},\aaa_{T,tf})$ form a torsion pair in $\aaa_T$ (see also \cite[Remark 6.16]{bayer2021stability}).
\end{rmk}
\begin{lm}\label{lmtf}
Let $E \in \aaa_C$, $C$ a DVR essentially of finite type over $R$ with uniformizer $\pi\in C$. 
If $E$ is torsion free then it is $C$-flat. 
\end{lm}
\begin{proof}
For the generic point $\xi\in \spec C$ the pullback functor $\iota_\xi^*$ is exact. So $E$ is $C$ flat if and only if $E_c\in\aaa_p$ for $p:=\spec C/\pi$ the closed point. By Lemma \ref{lmcurve} this happens if and only if $(\pi)\otimes E\ra E$ is injective in $\aaa_C$. If the map is not injective, then \eqref{eqlmcurve} shows also that $E$ is not torsion free. 
\end{proof}
\begin{lm}\label{lmlurie}
     Let $E, Q \in \mathcal{D}$ be relatively flat objects. Then the functor$$H^{0}\hhom_{R}(E,Q\boxtimes(-)):(R\mhyphen\Mod)\rightarrow(R\mhyphen\Mod)$$is representable by a finitely presented object $\cal{F}_{E,Q} \in R\mhyphen \Mod$.
\end{lm}
\begin{proof}
Let $K := \mathcal{H}om_{R}(E,Q)$. Because $\mathcal{D}$ is proper over $R$ we have that $K\in\Perf(R)$. By the projection formula in Lemma \ref{lmprojformula}, for any $R$-module $M$ there is an isomorphism:
$$\hhom_{R}(E,Q\boxtimes M) \cong \hhom_{R}(E,Q)\otimes_{R}M = K \otimes_{R}M.$$ We claim first that $K \in D(R)^{\geq 0}$. The fiber of $K$ at any point $t \in \text{Spec}(R)$ gives $K \otimes_{R} k(t) \cong \mathcal{H}om_{\mathcal{D}_t}(E_t, Q_t)$. By hypothesis, both $E_t$ and $Q_t$ lie in the heart $\mathcal{A}_t$, hence $\operatorname{Ext}^{i}_{\mathcal{D}_t}(E_t, Q_t) = 0$ for all $i < 0$. Consequently, $K$ has vanishing negative cohomology on every fiber. By Nakayama's Lemma, this implies that $K$ can be represented by a bounded complex of finitely generated projective modules strictly in non-negative degrees. Next, consider the dual object $K^{\vee} := \hhom_{R}(K, R)\in D(R)^{\leq 0}$. Since $K$ is perfect,  $K \otimes_{R}M \cong \hhom_{R}(K^{\vee}, M). $  As a result, the functor in the statement can be expressed as $H^{0}(K \otimes_{R} M) \cong \Hom_{D(R)}(K^{\vee}, M).$ To evaluate the latter, we use the following canonical triangle in $\Perf(R)$ with respect to the standard $t$-structure $$\tau_{std}^{\leq -1}K^{\vee} \rightarrow K^{\vee} \rightarrow H^{0}(K^{\vee}) \xrightarrow{[1]}$$ Applying the functor $\hhom_R(-, M)$ and looking at the induced long exact sequence, we have that $\Hom_{D(R)}(\tau_{std}^{\leq-1}K^{\vee}, M) = 0$ and $\Hom_{D(R)}((\tau^{\leq-1}K^{\vee})[1], M) = 0$. This vanishing forces an isomorphism $$\Hom_{D(R)}(K^{\vee}, M) \cong \Hom_{D(R)}(H^{0}(K^{\vee}), M)$$ and latter computes $\Hom_{R\mhyphen\Mod}(H^{0}(K^{\vee}), M)$. Therefore, we can set $\mathcal{F}_{E,Q} := H^{0}(K^{\vee})$.
\end{proof}
\begin{prop}\label{prophom}
    Let $E,\,Q\in\dd$ be relatively flat objects. The functor 
    \begin{align*}
        \underline{\Hom}_R(E,Q):(R\alg)&\rightarrow(\set)\\
        T&\mapsto H^0(\hhom_T(E_T,Q_T))
    \end{align*}
    is representable by an abelian cone of finite presentation over $\spec (R)$.
\end{prop}
\begin{proof}
    We verify directly our claim using a variant of \cite[Theorem D]{hall2014cohomology}. For $T\in(R\alg)$ and a morphism $\spec T\xrightarrow{\alpha}\spec R$, denote $\alpha_T^*:\dd\rightarrow \dd_T$ and $(\alpha_T)_*:\dd_T\rightarrow\ind(\dd)$ the corresponding functors. Thus, we have 
    \begin{align*}
        \underline{\Hom}_R(E,Q)(T)&=H^0\hhom_T(\alpha_T^*(E),\alpha_T^*(Q))\\
        &\cong H^0\hhom_R(E,(\alpha_T)_*\alpha_T^*(Q))\\
        &\cong H^0\hhom_R(E,Q\boxtimes\oo_T)\\&
        \cong\Hom_{\oo_R}(\mathcal{F}_{E,Q},\alpha_* \oo_T)\quad\hbox{by Lemma (\ref{lmlurie})}\\
        &\cong \Hom_{R\alg}(\sym_{\oo_R}^\bullet\mathcal{F}_{E,Q},\alpha_*(\oo_T))\\
        &\cong \Hom_{R\mhyphen\sch}(\spec T,\underline{\spec}_{\oo_R}(\sym_{\oo_R}^\bullet\mathcal{F}_{E,Q})).
    \end{align*}
    In particular, $\underline{\spec}_{\oo_R}(\sym_{\oo_R}^\bullet\mathcal{F}_{E,Q})$ is the representing object and we are done.
\end{proof}
\subsection{Quot functors}
Let $\sigma$ be a locally constant stability condition. For every $t\in\spec R$ we have a sluicing of width 1 on $\dd_t$ (or equivalently, a stability condition); denote by $\aaa_t$ the corresponding heart and by $H_{\aaa_t}^\bullet$ the corresponding cohomology functor. Let $E\in\dd$ be an $R$-flat object. 
\begin{defi}
    Define the Quot functors $$\quot_R^{\sigma}(E):(R\alg)\rightarrow (\set)$$
    sending an algebra $T\in(R\alg)$ to the groupoid of morphisms $E_T\rightarrow Q$ in $\dd_T$ such that 
    \begin{itemize}
        \item $Q\in\dd_T$ is $T$-flat with respect to $\sigma$,
        \item $E_t=H_{\aaa_t}^0(E_t)\rightarrow H_{\aaa_t}^0(Q_t)=Q_t$ is surjective for all $t\in \spec(T)$.
    \end{itemize}
\end{defi}
For what follows we will denote $\mm_{\sigma}:=\mm_{\sigma}{(0,1]}$. As in \cite{bayer2021stability}, we have the following.
\begin{prop}\label{propquot}
    If $\sigma$ is proper, then $\quot_R^{\sigma}$ is an algebraic space locally of finite type over $R$.
\end{prop}
\begin{proof}
    By Theorem (\ref{thmtv}) and Proposition (\ref{prophom}) we know that objects $E_{T'}\rightarrow Q'$ satisfy descent for an fppf morphism $T'\rightarrow T$, giving the existence of an object $[q_T:E_T\rightarrow Q]$. We need to check that $q_T\in\quot_R^{\sigma}(E)(T)$, that is $H^0(q_T)$ is surjective. Indeed, as pulling back along a field extentions is $t$-exact, we have that the surjectivity $H^0(q_T):H_{\aaa_t}^0(E_t)\twoheadrightarrow H_{\aaa_t}^0(Q_t)$ can be checked fppf-locally. Thus, $q_T\in\quot_R^{\sigma}(E)(T)$ and $\quot_R^{\sigma}(E)$ is a sheaf for the fppf topology.

    By the assumptions on $\sigma$, the stack $\mm_{\sigma}$ is a derived Artin stack locally of finite presentation over $R$. Thus, in order to prove the claim it is sufficient to show that the forgetful morphism $$\quot_R^{\sigma}(E)\rightarrow \mmm_{\sigma},$$
    sending a quotient $[q:E_T\rightarrow Q]\in\quot_R^{\sigma}(E)(T)$ to $Q\in\mmm_{\sigma}(T)$, is locally of finite presentation and representable by algebraic spaces. To this end, we claim that for an object $T\in(R\alg)$ together with a morphism $\spec(T)\rightarrow\mm_{\sigma}$ the projection $$Z:=\quot_R^{\sigma}(E)\times_{\mmm_{\sigma}}\spec (T)\rightarrow \spec (T)$$
    is a locally finitely presented morphism of algebraic spaces. To this end, observe that for $S$, an algebra over $R$, $$Z(S)=\{E_S\rightarrow Q_S\,|\,E_s\twoheadrightarrow Q_s\hbox{ is onto in }\aaa_s\,\forall\,s\in \spec( S) \},$$
    so that $Z\rightarrow \spec (T)$ factors through $Z\xrightarrow{\varphi} \underline{\Hom}_T(E_T,Q)$, the latter being an algebraic space of finite presentation by Proposition \ref{prophom}. To conclude it is enough to show that $\varphi$ is representable by open immersions. Given $S'\in(R\alg)$ and a morphism $\spec (S')\rightarrow \underline{\Hom}_T(E_T,Q)$, setting $Z':=Z\times_{\underline{\Hom}(E_T,Q)}\spec(S')$, we claim that $Z'\rightarrow\spec(S')$ is an open immersion. The morphism $\spec (S')\rightarrow \underline{\Hom}_T(E_T,Q)$ corresponds to a quotient $E_{S'}\rightarrow Q_{S'}$, and the subset $$U:=\{s'\in\spec(S')\,|\,E_{s'}\twoheadrightarrow Q_{s'}\hbox{ is onto}\}\subset S'$$
    represents $Z'\rightarrow\spec(S')$. By properness for $\sigma$ the subset $U\subset S'$ is open and the claim follows.
\end{proof}
\begin{rmk}
Due to Lemma \ref{lminstab}(2) and Definition \ref{defstabf}(a) we have a simpler description for the value of the Quot functors at a DVR $C$ over $R$, namely $$\quot_R^{\sigma}(E)(C)=\{\hbox{Quotients }E_C\twoheadrightarrow Q\hbox{ in }\aaa_C\,|\,Q\hbox{ is }C\hbox{-flat}\}.$$
\end{rmk}
\begin{prop}
Let $C$ be a DVR with field of fractions $K$, and consider a morphism $\spec(C)\rightarrow \spec (R)$.  Then for a flat object $E$ the morphism $\quot_R^{\sigma}(E)\rightarrow \spec (R)$ satisfies the existence part of the valuative criterion for properness with respect to $C$, provided that $\aaa_C$ has a $C$-torsion theory.
\end{prop}
\begin{proof}
We have to show that given $[\varphi:E_K\twoheadrightarrow Q_K]\in\quot_R^{\sigma}(E)(K)$ there exists a surjective morphism $\hat{\varphi}:E_C\twoheadrightarrow Q_C$ in $\aaa_C$ extending $\varphi$. Corollary \ref{coropen} together with $t$-exactness for the pullback functor induced by an open embedding grants the existence of such a surjective $\hat{\varphi}$. To conclude one needs to show that $Q_C$ is flat. By the assumption on $\aaa_C$ having a $C$-torsion theory, up replacing $Q_C$ with $Q_C^{tf}$ one can assume that $Q_C$ is torsion-free. Thus, we conclude by Lemma \ref{lmtf} that $Q_C$ can be chosen to be flat.
\end{proof}

We report a key lemma from \cite{bayer2021stability}.

\begin{lm}[\cite{bayer2021stability}, Lemma 11.21]\label{lmclosed}
Let $f:X\rightarrow Y$ be a morphism of algebraic spaces such that
\begin{itemize}
  \item $Y$ is Nagata,
  \item $f$ is of finite type,
  \item $f$ satisfies the strong existence part of the valuative criterion of properness with respect to any essentially locally of finite type morphism $\spec(C)\rightarrow Y$, $C$ a DVR.
\end{itemize}
Then $f$ is universally closed.
\end{lm}
\begin{defi}
For $\phi\in(0,1)$, we denote $\quot_R^{\sigma,\phi}(E)$ the subfunctor of $\quot_R^{\sigma}(E)$ assigning to $\spec(T)\rightarrow\spec (R)$ the set $$\quot_R^{\sigma,\phi}(E)(T):=\{E_T\twoheadrightarrow Q\in\quot_R^{\sigma}(E)(T)\,|\,\phi^+(Q_t)\geq\phi\hbox{ for all }t\in\spec(T)\}.$$
\end{defi}
\begin{prop}\label{propuniclosed}
Assuming properness for $\sigma$, $\quot_R^{\sigma, \phi}(E)$ is an algebraic space of finite type over $\spec(R)$ and the structure morphism $\quot_R^{\sigma, \phi}(E)\rightarrow\spec(R)$ is universally closed.
\end{prop}
\begin{proof}
The proof follows that of \cite[Lemma 21.21]{bayer2021stability}. Due to Definition \ref{defstabf}(2), since the central charge factors through a locally constant function $v:|\mm|\rightarrow\Lambda$ the forgetful morphism $$\quot_R^{\sigma,\phi}(E)\rightarrow\quot_R^{\sigma}(E)$$
is representable by open immersions, the latter being an algebraic space locally of finite type over $R$. The function $\phi_E^-$, being constructible on  the noetherian topological space $\spec(R)$, has a minimum $\phi_0>0$. As a result, for any $T$-point $E_T\twoheadrightarrow Q$ the phase of every stable factor of the restriction $Q_t$ is $\geq\phi_0$. Observe also that we have $\Im Z(Q_t)\leq\Im Z(E_t)$ and $\phi(Q_t)\geq \phi$ for every $t\in T$ and that this construction is compatible with base change. This means that the central charge of any stable factor of $Q_t$, for all $t\in T$, lies in the parallelogram in the complex plane with angles $\pi\phi_0$ and $\pi$ and vertices $0$ and $z$, defined as the intersection between the ray \{$\R_{\geq0} e^{i\pi\phi}\}$ and the line $\{\Im(z)=\Im(Z(E))\}$. By the support property and Remark \ref{rmksuppprop}(b), there is a finite set $\Gamma\subset\Lambda$ for the possible classes of the stable factors of $Q_t$ for $t\in T$. Since the central charge of $Q_t$ lies in the triangle with vertices $0,\,z$, and $w$, where $w$ is the complex number of phase $\phi_0$ such that $\Im(w)=\Im(Z(E))$, the set of sums of classes in $\Gamma$ with central charge lying in this triangle is another finite set $\Gamma'\subset\Lambda$. Thus, the forgetful morphism $\quot_R^{\ssigma,\phi}(E)\rightarrow\mm_{\ssigma}$ actually factors as $$\quot_R^{\sigma,\phi}(E)\xrightarrow{\iota}\bigsqcup_{v\in\Gamma'}\mm_{\sigma}^{(\phi_0,1]}(v)\rightarrow\mm_{\sigma}.$$
As in Proposition \ref{propquot}, the morphism $\iota$ is of finite type, and since $\mm_{\sigma}{(\phi_0,1]}(v)$ is finite type over $R$ for each $v\in\Lambda$ by Theorem \ref{thmhlr}(2), we get that $\quot_R^{\sigma,\phi}(E)$ is such.
The morphism $\quot_R^{\sigma}(E)\rightarrow\spec(R)$ satisfies the existence part of the valuative criterion for properness, and since $\sigma$ is locally constant on $\mm$ so does $\quot_R^{\sigma, \phi}(E)\rightarrow\spec(R)$. Finally, we conclude by  applying Lemma \ref{lmclosed} to the morphism $\quot_R^{\sigma, \phi}(E)\rightarrow\spec(R)$.
\end{proof}
\subsection{Relation with Harder--Narasimhan structures}
In view of Theorem \ref{thmmain}, we want to specify which hypotheses on the datum of a locally constant stability condition grant the existence of a sluicing on $\dd_C$, for $C$ a DVR over $R$ verifying item (a) of Definition \ref{defstabf}. To be more precise, we want to prove the following.
\begin{thm}\label{thmluicing}
    Let $\dd$ be a smooth and proper linear category over a Noetherian $k$-algebra $R$ of finite global dimension. Let $C$ be a DVR whose fraction field $K$ and residue field $C/\mathfrak{m}_C$ are residue fields of $R$. Suppose we have two slicings
    \begin{equation}\label{eqhn}
    \{\cal P_K(\phi)\}_{\phi\in\R} \hbox{ on }\dd_K\quad\hbox{and}\quad\{\cal P_{C/\mathfrak{m}_C}(\phi)\}_{\phi\in\R}\hbox{ on }\dd_{C/\mathfrak{m}_C},
    \end{equation}
    a locally constant $v:|\mm|\ra\Lambda$, and $Z:\Lambda\ra\C$ as in Proposition \ref{propsluicingstab} such that openness flatness and openness of stability hold. Suppose moreover that there exists a heart of a $t$-structure $\cal A_C$ on $\dd_C$ whose restriction to $\dd_K$ and $\dd_{C/\mathfrak{m}_C}$ equals $\cal P_K(0,1]$ and $\cal P_{C/\mathfrak{m}_C}(0,1]$ respectively. Then there exists a unique slicing $\{\cal P_C(\phi)\}_{\phi\in\R}$ on $\dd_C$ which restricts to the ones in \eqref{eqhn}.
\end{thm}
Recall first that by Remark \ref{rmksuppprop} a locally constant $v:|\mmm|\to\Lambda$ extends naturally to a group homomorphism $v:K_0(\dd_\kappa)\to\Lambda$ for each residue field $\kappa$ of $R$. Let $p:=\spec(\kappa)\rightarrow\spec R$ be the corresponding point, and observe that for any infinitesimal thickening $W$ of $p$ there is an isomorphism of Grothendieck groups $K_0(\dd_W)\xrightarrow{\sim}K_0(\dd_p).$ In particular, there is an immediate extension of $v$ given by
\begin{align*}
    v_W:K_0&(\dd_W)\to\Lambda\\
    E&\mapsto\ell(W)\,v(E),
\end{align*}
where $\ell(W)$ is the length of $W$. In particular, if $E=i_{W*}(E')$, we define $v_{tor}(E):=v_W(E')$. Thus, if $C$ is a DVR over $R$ such that its fraction field $K$ and its residue field $C/\mathfrak{m}_C$ are both residue fields of $R$ we have two induced group homomorphisms 
$$v_K:K_0(\dd_K)\to \Lambda\quad\text{and}\quad v_{tor}:K_0(\dd_{tor})\to\Lambda,$$
such that for every $E\in\dd_C$ and for each infinitesimal thickening $W$ of $p:=\spec(C/\mathfrak{m}_C)$ 
$$v_K(i_K^*E)=\frac{1}{\ell(W)}v_{tor}(i_{W*}i_W^* E).$$
Let $\aaa_C$ be the heart of a nondegenerate $t$-structure on $\dd_C$. Define the following central charge $$Z_C(E):=\begin{cases}
Z(v_K(i_K^*(E)) & \text{ if } i_K^*(E)\neq0\\
Z(v_{tor}(E)) & \text{ otherwise}.
\end{cases}$$
As a result, for $0\neq E\in\dd_C$ we have a notion of \emph{slope} given by 
$$\mu_C(E):=\begin{cases}
    -\frac{\Re(Z_C(E))}{\Im(Z_C(E))}&\text{ if } \Im(Z_C(E))\neq 0 \\
    +\infty&\text{ otherwise},
\end{cases}$$
and that of a \emph{phase} $\phi_C(E):=\frac{1}{\pi}\arg(Z_C(E))\in(0,1]$.
\begin{defi}
    An object $E\in\aaa_C$ is said to be $C$-\emph{semistable} if for all proper subobjects $0\neq F\hookrightarrow E$ we have $\mu_C(F)\leq\mu_C(E)$. 
\end{defi}
\begin{rmk}
    By the see-saw property of $Z_C$ an object $0\neq E\in\aaa_C$ is $C$-semistable if and only if for all proper $0\neq A\hookrightarrow E$ we have $\mu_C(A)\leq\mu_C(E/A)$.
\end{rmk}
\begin{defi}
    The function $Z_C$ has the Harder--Narasimhan property (HN property) if every $0\neq E \in \aaa_C$ admits a unique filtration $$0=E_0\hookrightarrow E_1\hookrightarrow\dots\hookrightarrow E_m=E$$
    such that for every $i=1,\dots,m$ the object $E_i/E_{i-1}$ is $C$-semistable and $$\phi_C(E_{i+1}/E_i)>\phi_C(E_i/E_{i-1}).$$
\end{defi}
\begin{prop}\label{propheart}
The datum of a local heart $\aaa_C$ of a bounded $t$-structure on $\dd_C$ together with a central charge $Z_C=(Z_K,Z_{tor})$ with the HN property on $\aaa_K$ and $\aaa_{tor}$ respectively is equivalent to that of a local slicing $\{\cal P_C(\phi)\}_{\phi\in\R}$ on $\dd_C$ such that for $0\neq E\in\cal P_C(\phi)$ we have that either $i_K^*(E)\neq 0$ and $Z_K(i_K^*(E))\in\R_{>0}e^{i\pi\phi}$ or $E\in \dd_{C\mhyphen tor}$ and $Z_{tor}(E)\in\R_{>0}e^{i\pi\phi}$.
\end{prop}
\begin{proof}
    A local slicing $\cal P_C$ on $\dd_C$ gives a $C$-local heart $\aaa_C:=\cal P_C(0,1]$. Clearly, the pair $Z_C:=(Z_K,Z_{tor})$ defines a stability function on $\aaa_C$. To conclude, it suffices to observe that the existence of HN filtration with respect to the slicing $\cal P_C$ gives the HN property for $Z_C$. Let $E\in\cal P_C(\phi)$ for $\phi\in(0,1]$. The existence of a maximal destabilizing subobject for $E$ with respect to $Z_C$, together with the locality of $\cal P_C$, would contradict the existence of HN filtration in either $\aaa_K$ or $\aaa_{tor}$. Conversely, for $\phi\in(0,1]$ define $\cal P_C(\phi)$ to be the full subcategory of $Z_C$-stable objects $E$ of phase $\phi_C(E)=\phi$ and $$\cal P_C(m+\phi):=\cal P_C(\phi)[m]$$
    for $m\in\Z$. The remaining part of the proof goes exactly as in \cite[Lemma 5.3]{bridgeland2007stability}
\end{proof}
\begin{rmk}
The data in Proposition \ref{propheart} is called a \emph{Harder--Narasimhan} structure, defined originally in \cite{bayer2021stability}. It is part of the definition of a \emph{flat family of stability conditions} introduced in that work. The approach is quite similar to that of \cite{halpernleistner2025spaceaugmentedstabilityconditions} - it will turn out at the end of this section that the two are equivalent.
\end{rmk}
We have the following version of Langton's semistable reduction theorem. Observe that no assumptions are needed other than openness of stability and generic flatness for the heart $\aaa_C$. In particular, we do not assume our stack $\mm_\sigma$ to admit a $\Theta$-stratification by Harder--Narasimhan filtrations yet.
\begin{thm}[Semistable reduction]\label{thmssred}
    Let $C$ be a DVR over $R$. Denote $K$ and $\kappa=C/\mathfrak{m}_C$ its field of fractions and its residue field respectively, and let $(\pi)=\mathfrak{m}_C$ be a uniformizer. Assume further that openness of stability and generic flatness holds for $\aaa_C$. Let $E_K$ be a semistable object in $\aaa_K$. There exists a $C$-semistable object $E\in\aaa_C$ such that $i_K^*(E)\cong E_K$.
\end{thm}
\begin{proof}
    By Lemma \ref{lmloc} there exists $F\in\dd_C$ such that $i_K^*F\cong E$, and by $t$-exactness of $i_K^*$ the object $F$ must lie in $\aaa_C$. Let $Q$ be the minimal destabilizing quotient of $i_\kappa^*(F)$, which exists by hypothesis. Namely, we have a short exact sequence 
    \begin{equation}\label{eqlang}
    0\rightarrow i_{\kappa*}D\rightarrow i_{\kappa*}i_\kappa^*F\rightarrow i_{k*}Q\rightarrow 0
    \end{equation}
    in $\aaa_C$. Define $F':=\ker(F\twoheadrightarrow i_{\kappa*}i_\kappa^*F\twoheadrightarrow i_{\kappa*}Q)\subset F\in\aaa_C$, which is well defined by $t$-exactness of $i_\kappa$. Observe furthermore that $i_K^*F'\cong i_K^*F\cong E$ and $F/F'\cong i_{\kappa*}Q$. Since $F/\pi F\cong i_{\kappa*}i_\kappa^* F$ and $\pi\cdot F\subset F'\subset F$, \eqref{eqlang} can be rewritten as $$0\rightarrow i_{\kappa*} D\rightarrow F/\pi\cdot F\rightarrow F/F'\rightarrow0,$$
    which in turn implies $i_{\kappa*}D\cong F'/\pi\cdot F$. Moreover, since $F'/\pi\cdot F'\cong i_{\kappa*}i_{\kappa}^*F'$ we have a surjection $i_{\kappa*}i_{\kappa}^*F'\twoheadrightarrow F'/\pi\cdot F\cong i_{\kappa*}D$, hence the following short exact sequence 
    \begin{equation}\label{eqlang2}
        0\rightarrow i_{\kappa*}Q\rightarrow i_{\kappa*}i_{\kappa}^*F'\rightarrow i_{\kappa*}D\rightarrow 0.
    \end{equation}
    Again, we have the short exact sequence $$0\ra i_{\kappa*}D'\ra i_{\kappa*}i_\kappa^*F'\ra i_{\kappa*}Q'\ra0,$$
    where $D'$ is the last term in the HN filtration of $i_{\kappa}^*E'$. 
    
    Consider the Harder--Narasimhan polygons $\HNP(i_\kappa^*F)$ and $\HNP(i_\kappa^*F')$. We first claim that there is an inclusion $\HNP(i_\kappa^*F')\subset\HNP(i_\kappa^*F)$. Indeed, if $S\hookrightarrow i_\kappa^* F'$ is a subobject, we have a commutative diagram as follows
    $$\begin{tikzcd}
        0\arrow[r]&T\arrow[r,hook]\arrow[d,hook]&S\arrow[r,two heads]\arrow[d,hook]&P\arrow[r]\arrow[d,hook]&0\\
        0\arrow[r]&D'\arrow[r,hook]& i_\kappa^*F'\arrow[r, two heads]& Q'\arrow[r]&0.
    \end{tikzcd}
    $$
    Since $Q'$ is a subobject of $i_\kappa^*F$, $Z_\kappa(P)\in\HNP(i_\kappa^*F)$. Moreover, since the object $D'$ is semistable, $\mu_\kappa(T)\leq\mu_\kappa(D')$. Observing that $Z_\kappa(S)=Z_\kappa(P)+Z_\kappa(T)$ we have that $Z_\kappa(S)\in\HNP(i_\kappa^*F)$, hence $\HNP(i_\kappa^*F')\subset\HNP(i_\kappa^*F)$. We claim next that if $\HNP(i_\kappa^*F')=\HNP(i_\kappa^*F)$ then the short exact sequence in \eqref{eqlang2} is split. Let $K:=Q\times_{i_\kappa^*F'} D'$ be the kernel of the composition $Q\ra i_\kappa^* F'\ra Q'$, and observe that it sits in the short exact sequence $$0\ra K\ra D'\ra D'/K\ra0.$$ 
    Now, if $K\neq 0$, by semistability of $Q$ we have the inequalities $$ \mu_\kappa(K)\leq\mu_\kappa(Q)<\mu_\kappa(i_\kappa^* F)=\mu_\kappa(i_\kappa^*F')\leq\mu_\kappa(D').$$ 
    By the see-saw property for $Z_\kappa(-)$, it follows that $\mu_\kappa(D'/K)\geq\mu_\kappa(D')$. By construction, from $Z(D'/K)\in\HNP(i_\kappa^*F')$ one gets that $\Re(Z(D'/K))<\Re(Z(D'))$. Thus, $\mu_k(K)>\mu_k(Q)$, contradicting the stability of $Q$. As a result, $K=0$. This observation, together with the fact that $Z(Q)=Z(Q')$, implies that $Q\cong Q'$ is an isomorphism splitting \eqref{eqlang2}. By the support property, there are finitely many Harder--Narasimhan polygons contained in $\HNP(i_\kappa^*F)$. In particular, if $i_\kappa^*F'$ is not stable, up to iterating the argument, we end up with an infinite filtration of $F$
    $$\dots\subset F^{(n)}\subset\dots\subset F^{(2)}\subset F^{(1)}\subset F^{(0)}= F$$
    whose terms $i_\kappa^*F^{(n)}$ have the same Harder--Narasimhan polygon for all $n\geq0$. As shown in the previous discussion, each restriction is split, i.e. $i_\kappa^*F^{(n)}\cong Q\oplus D$ for every $n\geq0$, so every quotient $F/F^{(n)}$ satisfies $i_\kappa^*(F/F^{(n)})\cong Q$. Looking at the following commutative diagram with exact rows
    \begin{equation}\label{eqlangton}
    \begin{tikzcd}
& (F/F^{(n+1)}) \otimes_{C/\pi^{n+1}} \pi^n \arrow[r] \arrow[d, "\cong"] 
    & F/F^{(n+1)} \arrow[r] \arrow[d, "\id"] 
    & (F/F^{(n+1)}) \otimes_{C/\pi^{n+1}} C/\pi^n \arrow[r] \arrow[d, dashed, "\cong"'] 
    & 0\\
0\arrow[r]&F^{(n)}/F^{(n+1)} \arrow[r] 
    & F/F^{(n+1)} \arrow[r] 
    & F/F^{(n)} \arrow[r] 
    & 0
\end{tikzcd}
\end{equation}
we end up with an isomorphism $F/F^{(n)}\cong (F/F^{(n+1)})\otimes_{C/\pi^{n+1}}C/\pi^n$ by the five-lemma. Next, we claim that $F/F^{(n+1)}$ is a flat object over $C/\pi^{n+1}$. We proceed by induction, the case $n=0$ holding by Lemma \ref{lminstab}(a). The inductive case is achieved as an application of Lemma \ref{lmlocalflatness}. Indeed, by the previous discussion $(F/F^{(n+1)})\otimes_{C/\pi^{n+1}}C/\pi^n\cong F/F^{(n)}$ is flat by the inductive hypothesis. Moreover, the natural morphism $$(F/F^{(n+1)}) \otimes_{C/\pi^{n+1}}  (\pi^n/\pi^{n+1})\ra F/F^{(n+1)}$$ corresponds to the first map in the second row of \eqref{eqlangton}, which is injective.

By assumption, since openness of flatness holds for $\aaa_C$ the Quot functor $\quot_C^{\sigma,\phi}(F)$ gives an algebraic space locally of finite type over $C$. Thus, we end up with a sequence of compatible lifts 
$$
\begin{tikzcd}
    \spec\kappa\arrow[r,"(i_\kappa^*F\twoheadrightarrow Q)"]\arrow[d, hook] & \quot_C^{\sigma,\phi}(E)\arrow[d]\\
    \spec(C/\pi^n)\arrow[r, hook]\arrow[ur, dashed]&\spec C.
\end{tikzcd}
$$
In particular, there exists an extension of DVR's $C\ra C'$ and an object $(F_{C'}\twoheadrightarrow \tilde{Q})\in\quot_C^{\sigma,\phi}(F)(C')$. Restricting to the fraction field $K'$, by exactness of $i_{K'}^*$, we have a destabilizing quotient $i_{K'}^*F\twoheadrightarrow i_{K'}^*\tilde{Q}$, contradicting the stability of $E_K$.
\end{proof}

\begin{lm}\label{lmtorsiontheory}
    Keep the hypotheses as above, then $\aaa_C$ admits a $C$-torsion theory.
\end{lm}
\begin{proof}
    It is enough to prove that given $E\in\aaa_C$ there exists $m\in\Z_{\geq0}$ such that $\pi^m\cdot E$ is torsion-free. If that is the case, then $$T:=\ker(E\twoheadrightarrow(\pi^m\cdot E)\otimes(\pi)^{-m})$$ 
    is the maximal torsion subobject of $E$, hence $\aaa_C$ has a $C$-torsion theory. For every $j\geq0$, consider the short exact sequence 
    \begin{equation}\label{equniform}
    0\ra (\pi)^{j+1}\cdot E\ra (\pi)^j\cdot E\ra (\pi)^j\cdot E/(\pi)^{j+1}\cdot E\ra0.
    \end{equation}
    By Lemma \ref{lmcurve}, applying $i_{\kappa*}i_\kappa^*$ to \eqref{equniform} and taking the long cohomology sequence with respect to the heart $\aaa_C$ we have an exact sequence of the form 
    \begin{equation}\label{equniform2}
        0\ra\Ann((\pi)^{j+1};E)\ra\Ann((\pi)^j;E)\xrightarrow{\alpha_j} (\pi)^j\cdot E/(\pi)^{j+1}\cdot E\ra (\pi)^{j+1}\cdot E/(\pi^{j+2}\cdot E)\ra 0.
    \end{equation}
    In particular, by composing the last arrow for $j\geq0$ there is an infinite sequence of surjections in $\aaa_C$ as follows 
    \begin{equation}\label{equniform3}
    E/(\pi)\cdot E\twoheadrightarrow (\pi)\cdot E/(\pi)^2\cdot E\twoheadrightarrow\dots\twoheadrightarrow(\pi)^j\cdot E/(\pi)^{j+1}\cdot E\twoheadrightarrow(\pi)^{j+1}\cdot E/(\pi)^{j+2}\cdot E\twoheadrightarrow\dots
    \end{equation}
    Next, we show that the sequence above must stabilize. Indeed, consider the objects
    \begin{equation*}
    \begin{split}
        K_j&:=\ker(E/\pi\cdot E\twoheadrightarrow (\pi)^{j+1}\cdot E/(\pi)^{j+2}\cdot E),\\
        Q_j&:=\coker(\Ann((\pi)^{j+1};E)\ra\Ann((\pi);E)),\\
        I_j&:=\im(\alpha_j).
    \end{split}
    \end{equation*}
    By Lemma \ref{lmcurve}, $\Ann((\pi);E)\cong H_{\aaa_\kappa}^{-1}(i_\kappa^*E)$, so we have that $\mu_\kappa(Q_j)\geq\mu_\kappa^-(H_{\aaa_\kappa}^{-1}(i_\kappa^*E))$ for all $j\geq0$. Moreover, we have the two short exact sequences $$0\ra K_j\ra K_{j+1}\ra I_{j+1}\ra 0\quad\text{and}\quad0\ra I_{j+1}\ra Q_{j+1}\ra Q_j\ra0$$
    together with an isomorphism $Q_0\cong I_0\cong K_0$, so that by induction $K_j$ and $Q_j$ have the same class in $K_0(\aaa_\kappa)$ and the claim follows by the support property. 

    Up to replacing $E$, we may assume that all the maps in \eqref{equniform3} are isomorphisms. Hence, by Lemma \ref{lmlocalflatness2} $F_j:=\Ann((\pi)^j;E)$ is flat over $C/(\pi)^j$ for all $j\geq0$. Lastly, we claim that $F_j\neq0$ leads to a contradiction. Since we assume $\sigma$ to be proper, the heart $\aaa_C$ verifies openness of flatness and the stack $\mm_\sigma$ is algebraic and locally of finite presentation. Thus, the objects $F_j$ correspond to a sequence of compatible morphisms $$\spec(C/(\pi)^j)\ra\mm_\sigma,\quad j\geq0.$$
    By Artin approximation, there exists an extension of DVR's $C\ra C'$ and an object $F\in\mm_\sigma(C')$ restricting to $F_j$ at each $C/(\pi)^j$. Consider the mapping object $\cHom_{C'}(F,E_{C'})\in\Perf(C')$. Then, as $E,\, F\in\aaa_C$, we have that 
    \begin{equation}\label{equniform4}
    \Hom_{\aaa_{K(C')}}(F_{K(C')},E_{K(C')}[-1])\cong H_{std}^{-1}(K(C')\otimes\cHom_{C'}(F,E_{C'}))=0,
    \end{equation}
    where $H_{std}^\bullet(-):D_{qc}(C'\mhyphen\Mod)\ra C'\Mod$ denotes the cohomology functor with respect to the standard $t$-structure on $D_{qc}(C'\mhyphen\Mod)$. However, by construction and Lemma \ref{lmcurve} we have $$H_{std}^{-1}(R/(\pi)^j\otimes\cHom_{C'}(F,E_{C'}))\cong \Hom_{\aaa_{C}}(F_j, F_j)\neq 0,$$
    contradicting \eqref{equniform4}.
\end{proof}
\begin{prop}\label{lmtor}
    Assume the pair $(\aaa_p,Z_p)$ is a stability condition on $\dd_p$. Then the function $Z_{tor}:=Z\circ v_{tor}$ on $\aaa_{tor}$ has the HN property.
\end{prop}
\begin{proof}
   Given $E\in\aaa_{tor}$ , we proceed by induction on the length $n$ of the support of $E$, the case  $n=1$ holding by assumption. Both $E/\pi\cdot E$ and $(\pi\cdot E)\otimes(\pi)^\vee$ are quotients of $E$ and the length of their support is strictly less than that of $E$, hence they admit HN filtrations. Choose $Q_0$ among the maximal destabilizing quotients of $E/\pi\cdot E$ and $(\pi\cdot E)\otimes(\pi)^{-1}$ to be the one of phase $$\phi(Q_0)=\min\{\phi_{tor}^-(E/\pi\cdot E),\phi_{tor}^-((\pi\cdot E)\otimes(\pi)^{-1})\}.$$
   By the see-saw property, $Q_0$ is of minimal phase among quotients of $E$. Arguing as above, we can find a semistable quotient $Q_2$ of $E_i:=E/Q_0$ of minimal phase, and letting $E_2:=E_1/Q_1$ we see that $\phi(E_1/E_2)\geq\phi(E/E_1)$ by the see-saw property. In this way we end up with a decreasing filtration 
   $$\dots\subset E_i\subset E_{i-1}\subset\dots\subset E_2 \subset E_1\subset E,$$
   with semistable factors of ascending slope. The central charge $Z_{tor}$ of any quotient lies within the parallelogram with edges of angles $\pi$ and $\pi\phi^-(E)$ end with endpoints $0$ and $Z_C(E)$. By the support property, there are finitely many classes for $v_{tor}(E_i/E_{i+1})$ and and an index $\hat{i}$ with $\phi(E_i/E_{i+1})=1$ for all $i\geq\hat{i}$. This means that $E_{\hat{i}}$ is itself semistable with $\phi(E_{\hat{i}})=\pi$. As a result, the filtration terminates and, up to identifying quotients of the same slope, $E_\bullet$ gives the HN filtration of $E$.
\end{proof}

\begin{cor}\label{corHN}
 If the heart $\aaa_C$ satisfies openness of stability the central charge $Z_C=(Z_K,Z_{tor})$ has the HN property.
\end{cor}
\begin{proof}
    Given $E\in\aaa_C$ we can construct directly a maximal destabilizing subobject. Let $\mu:=\mu_K^+(E)$ be the maximal slope of $i_K^*E$. If $E_{tor}\neq 0$ and $\mu_{tor}^+(E)>\mu$ then $F_{tor}\hookrightarrow E_{tor}$, the maximal destabilizing subobject of $E_{tor}$, is a maximal destabilizing subobject for $E$. Indeed, for a subobject $S\hookrightarrow E$ either $\mu_K(i_K^*S)\leq\mu$ or $S\hookrightarrow E$ factors through $E_{tor}$ and we have that $\mu_{tor}(S)\leq\mu_{tor}(F_{tor})$ by construction. 
    
    If instead $E_{tor}=0$ or $\mu_{tor}^+(E_{tor})\leq \mu$, consider the maximal destabilizing subobject of $i_K^*E$ - by Lemma \ref{lmloc} we may assume that it is of the form $i_K^*F$ for $F\in\aaa_C$. Since $\aaa_C$ has a $C$-torsion theory, we may also assume that $E_{tor}\cong F_{tor}$. Let $F_C\in\aaa_C$ be the semistable extension of $F$ obtained by Theorem \ref{thmssred}. Since $\aaa_C$ admits a $C$-torsion theory, we may assume that $F_{tor}\cong E_{tor}$. The stability function $Z_{tor}$ on $\aaa_{tor}$ has the HN property, so the subcategories 
    \begin{align*}
        \cal T^\beta&:=\{E\in\aaa_{tor}\,|\,E\text{ is }\mu_{tor}\text{-semistable},\,\mu_{tor}(E)>\beta\},\\
        \cal F^\beta&:=\{E\in\aaa_{tor}\,|\,E\text{ is }\mu_{tor}\text{-semistable},\,\mu_{tor}(E)\leq\beta\}
    \end{align*}
    form a torsion pair in $\aaa_{tor}$. In particular, in the case $\beta=\mu$ it follows that there exists an object $F_C\subset\tilde{F}\subset F$ such that $\mu_C^+(F/\tilde{F})<\mu\leq \mu_C^-(\tilde{F}/F_C)$. Observe that $\tilde{F}$ must be semistable. Indeed, if $\mu_C(S)>\mu_C(\tilde{F}/S)$ for some subobject $S\hookrightarrow \tilde{F}$ then $S$ cannot be torsion, as it would contradict $\mu_{tor}^+(E_{tor})\leq \mu$, so $\tilde{F}/S$ must be torsion with $\mu>\mu_{tor}(\tilde{F}/S)$. In this case, in the exact sequence 
    $$ 0\ra F_C/(F_C\times_{\tilde{F}} S)\ra \tilde{F}/S\ra\tilde{F}/(F_C+S)\ra0$$
    we have that $\mu_C(F_C/(F_C\times_{\tilde{F}} S))\geq\mu$ by stability of $F_C$ and $\mu_C(\tilde{F}/(\tilde{F}+S))\geq \mu$, contradicting the see-saw property of $\mu_C$. In the exact sequence $$0\ra F/\tilde{F}\ra E/\tilde{F}\ra E/F\ra 0$$
    we have that by construction $\mu_C^+(F/\tilde{F})<\mu$ and that $E/F$ is torsion-free, so $\mu_C(T)=\mu_K(T)\leq\mu$ for every subobject $T\hookrightarrow E/F$. Lastly, we observe that the process of finding a maximal destabilizing quotient must terminate because at each step the length of the HN filtration of either $E_K$ or $E_{tor}$ must decrease.
\end{proof}
\subsubsection*{Proof of Theorem \ref{thmluicing}}
Since openness of stability holds by Corollary \ref{corHN}, the central charge $Z_C$ has the HN property. As a result, we may apply Proposition \ref{propheart}.
\qed
\\

We end the section with the following. 
\begin{prop}\label{propnoethcurve}
    Suppose that $Z:\Lambda\rightarrow\C$ factors through $\Q[i]$, then the heart $\aaa_C$ is a noetherian abelian category. 
\end{prop}
\begin{proof}
   By \cite[Proposition 5.0.1]{AbramovichPolishchuk+2006+89+130} we know the claim to hold for the cateories $\aaa_K$ and $\aaa_{tor}$. Let $E\in\aaa_C$, and consider an infinite flag $$E_0\hookrightarrow E_1\hookrightarrow E_2\hookrightarrow\dots\hookrightarrow E.$$
   Looking at the induced morphisms in $\aaa_K$ we see that the sequence stabilizes. As a result we may assume that each $F_i:=\coker(E_i\hookrightarrow E_{i+1})$ is a torsion object. The claim follows by noting that the cokernel sequence must stabilize as well by noetherianity of $\aaa_{tor}$.
\end{proof}
\subsection{The stability manifold}
We are ready to present the main result of this section, namely a structural statement for the set of locally constant stability conditions $\stab(\dd/R)$ which are proper. In this section we fix the topology on $\stab(\dd/R)$ to be the coarsest one such that the forgetful map
\begin{align*}
\stab(\dd/R)&\rightarrow\stab(\dd_t)\\
\sigma&\mapsto\sigma_t
\end{align*}
is continuous for every $t\in\spec(R)$.
\begin{thm}\label{thmmain}
The topological space $\stab(\dd/R)^\circ$ of \emph{proper} locally constant stability conditions, if non empty  has the structure of a complex manifold and the map $$\stab(\dd/R)^\circ\rightarrow\Hom(\Lambda,\C)$$
sending a locally constant stability condition to its central charge $Z:\Lambda\ra\C$ is a local isomorphism.\\
In particular, if $\sigma\in\stab(\dd/R)$ satisfies the conditions in item (2) of Theorem (\ref{thmhlr}) so do all $\sigma'$ in the connected component of $\stab(\dd/R)$ containing $\sigma$.
\end{thm}
Let $\sigma\in\stab(\dd/R)$. By Remark \ref{rmksuppprop}(b), for each field $k$ over $R$ there is an induced stability condition on $\dd_k$. Since we have a $\widetilde{GL}_2^+(\R)$-action on $\stab(\dd_{\kappa(s)})$ for all $s\in\spec (R)$ we can apply the reduction steps provided by \cite{bayer2019short}. As a result, we can treat only real variations of the central charge $Z:\Lambda\ra\C$. The upshot of this step is that the hearts $\cal P_{\kappa(s)}((0,1])$  for every $s\in\spec R$ and $\cal P_C((0,1])$ for every DVR $C$ over $R$ remain unchanged. This means that \emph{openness of flatness} will hold for the deformed stability conditions. 

Indeed, one can assume by \cite[Section 7]{bayer2019short} that the quadratic form $Q$ has signature $(2, \operatorname{rk}( \Lambda)-2)$. Let $||\cdot||$ be the norm associated to $-Q|_{\ker Z}$, and let $p:\Lambda_\R\ra\ker Z$ be the orthogonal projection with respect to $Q$. By \cite[Lemma 4.2]{bayer2019short}, up to the action of $\operatorname{GL}_2^+(\R)$ we can assume that $$Q(v)=|Z(v)|^2-||p(v)||^2.$$
As a consequence, one can only consider deformations of the central charge of the form $W=Z+u\circ p$, where $u:\ker Z\ra\R$ is a linear map with operator norm $||u||<\delta$. Under these hypotheses, we have that $|W(v)-Z(v)|\leq||u||\cdot|Z(v)|$. 

For each $s\in\spec R$, applying Theorem \ref{thmdefo} to the stability condition $\sigma_{\kappa(s)}=(Z,\cal P_{\kappa(s)})$, we get a stability condition $\tau_{\kappa(s)}:=(W,\cal Q_{\kappa(s)})\in\stab(\dd_{\kappa(s)})$. By Proposition \ref{propsluicingstab}, this gives the items (1), (2), and (3) of Definition \ref{defstabf}. We need to check the remaining conditions (a) and (b). 

Let $\phi\in(0,1)$, and set $\epsilon\in[0,1/2]$ to be such that $\frac{\sin(\pi\epsilon)}{\sin(\pi\phi)}=||u||$ so that $0<\phi-\epsilon<\phi+\epsilon<1$. 
\begin{lm}[\cite{bayer2021stability}, Lemma 22.3]\label{lmslicings}
    For all $s\in\spec R$ it holds that $\cal Q_{\kappa(s)}(\phi)\subset\cal P_{\kappa(s)}[\phi-\epsilon,\phi+\epsilon]$.
\end{lm}
\begin{proof}
    Consider an object $E\in\cal Q_{\kappa(s)}$, and consider its first Harder-Narasimhan factor $A\hookrightarrow E$. By the support property $Q(A)\geq 0$, and $$\phi\geq\phi(W(A))\geq \phi(Z(A)+||u||\cdot|Z(A)|.$$
    By our choice of $\epsilon$, the right hand side equals $\phi$ for $\phi(Z(A))=\phi+\epsilon$. It follows that $\phi^+(E)=\phi(Z(A))\leq \phi+\epsilon$. The same argument for the maximal destabilizing quotient gives $\phi^-(E)\geq\phi-\epsilon$.
\end{proof}
\subsubsection*{Proof of Theorem \ref{thmmain}} We prove the theorem in several steps. If $\stab(\dd/R)^\circ\neq\emptyset$ there exists a proper locally constant $\sigma\in\stab(\dd/R)$. 

Consider the collection $\tau:=\tau_{\kappa(s)}=(W,\cal Q_{\kappa(s)})\in\stab(\dd_{\kappa(s)})$, obtained as in the previous discussion. Since from the very construction of $\tau$ every $\tau_{\kappa(s)}$ satisfies the support property with respect to the same quadratic form as $\sigma$, condition (b) of Definition \ref{defstabf} follows automatically. 
    \subsubsection*{Step 1 - Openness deforms}  Since $\sigma$ verifies openness of stability, $\mmm_{\sigma}([\phi-\epsilon,\phi+\epsilon])$ is an open substack of $\mmm$. To conclude, it suffices to show that for every $v\in K_0(\dd)$ the substack $$\mmm_{\tau}^{ss}(v)\hookrightarrow\mmm_{\sigma}([\phi-\epsilon,\phi+\epsilon])$$ is open. Let $T\in R\alg$ and let $E\in\mmm_{\sigma}([\phi-\epsilon,\phi+\epsilon])(T)$; in particular $\phi^\pm(E_t)\in[\phi-\epsilon,\phi+\epsilon]$ for all $t\in \spec T$. By Lemma \ref{lmslicings}, any $W$-semistable quotient $E_t\ra Q$ such that $\phi(W(Q))\leq\phi$ satisfies $\phi(Z(Q))\leq\phi+\epsilon$ so gives an object of $\quot_R^{\sigma,\phi+\epsilon}(T)$. By the support property and the local constancy of $v:|\mmm|\ra\Lambda$ the condition $\phi(W(Q))\leq\phi$ picks out finitely many connected components of $\quot_R^{\sigma,\phi+\epsilon}$, and the union of their images in $T$ is exactly the unstable locus. By Proposition \ref{propuniclosed} and properness of $\sigma$, this locus is closed by universal closedness of the Quot functors.
    \subsubsection*{Step 2 - Boundedness is preserved} To see that boundedness of families of semistable objects of a fixed class $v\in\Lambda$ is preserved within $\stab(\dd/R)$ it suffices to observe that for $v\in\Lambda$ we have $\mmm_{\tau}(v)^{ss}\subset\mmm_{\ssigma}([\phi-\epsilon,\phi+\epsilon])(v)$, which is bounded by Corollary \ref{corinterval}. Thus, families of $\tau$-semistable objects of class $v$ are bounded if and only if they are bounded with respect to $\sigma$. 
    \subsubsection*{Step 3 - Sluicings deform} From the previous discussion, since we are considering only real variations of the central charge, we have that for any DVR $C$ over $R$ such that $K$ and $C/\mathfrak{m}_C$ are both residue fields of $R$ there is a local heart given by $\cal P_C(0,1]$, restricting to $\cal Q_K(0,1]$ and $\cal Q_{C/\mathfrak{m}_C}$ respectively. It follows by Theorem \ref{thmluicing} that there is a slicing on $\dd_C$, denoted as $\{\cal Q_C(\phi)\}_{\phi\in\R}$, inducing those on $\dd_K$ and $\dd_{C/\mathfrak{m}_C}$. We need to show that this datum defines a sluicing on $\dd_C$ as in Definition \ref{defstabf}(a). Indeed, let $w\leq1$. Since $\Lambda$ is assumed to have finite rank, we can choose $$\epsilon'<\min\biggl\{\frac{1-w}{2},\epsilon\biggr\}$$
    and a $W':\Lambda\ra\C$ factoring through $\Q[i]$. This induces a stability condition $\tau_{\kappa(s)}'=(\cal Q_{\kappa(s)}',W')$ on $\dd_{\kappa(s)}$ for all $s\in\spec(R)$ with noetherian $t$-structures. By Lemma \ref{lmslicings} this means that $$d(\cal Q_{\kappa(s)}',\cal Q_{\kappa(s)})<\epsilon'\quad\text{for all }s\in\spec(R).$$ As shown in the previous step, $\tau'$ verifies openness of flatness. Hence, Lemma \ref{lmtorsiontheory} implies that the central charge $W_C:=(W_K,W_{tor})$ induces a $C$-torsion theory on $\cal Q_C(\phi,\phi+1]$ and $\cal Q_C'(\phi,\phi+1]$ for every $\phi\in\R$. As a result, the bounds $d(\cal Q_K,\cal Q_K')<\epsilon'$ and $d(\cal Q_{C/\mathfrak{m}_C},\cal Q_{C/\mathfrak{m}_C}')<\epsilon'$ imply that $d(\cal Q_C,\cal Q_C')<\epsilon'$.  By Proposition \ref{propnoethcurve} the $t$-structures $(\cal Q_C'(>\phi),\cal Q_C'(\leq \phi))$ are noetherian. As in Remark \ref{rmksluicing}(iii), from the bounds above we have that  for all $0<b-a<w$ and $\psi\in(b-1+\epsilon', a-\epsilon')$ there are inclusions $$\cal Q_C(>b-1)\subset\cal Q_C'(>\psi)\subset\cal Q_C(>a).$$
    Hence, the collection of $t$-structures $\{\cal Q_C(>\phi),\cal Q(\leq\phi)\}$ defines a sluicing of width 1 on $\dd_C$. 
    \subsubsection*{Step 4 - Properness is preserved} Lastly, we need to show that condition (2) from Theorem \ref{thmhlr} is satisfied for $\tau$. We have seen that $\mm_\tau$ is locally of finite presentation (hence has affine diagonal by \cite[Lemma 7.20]{alper2023existence}), and that for every $v\in\Lambda$ the substack $\mm_\tau^{ss}(v)\subset\mm$ is open and bounded. Let $C>0$ be a positive constant. By the support property there are finitely many classes $w\in\Lambda$ such that $|Z(w)|<C$ and the stack $\mm_\tau^{ss}(w)$ is nonempty. As a result, the set $\cal S$ of semistable points in $|\mm|$ with phase in any interval $(a,b]$ and mass $< C$ is bounded. Again, by the support property we have also that $$\inf_{E\in \cal S} |Z(v(E))|=\min_{E\in \cal S}|Z(v(E))|>0.$$ 
    In particular, there exists an integer $N\in\Z_{\geq 0}$ such that the length of the Harder--Narasimhan filtration of any object $E$ with $m_\tau(E)<C $ is less than $N$. Consider the stack $\filt^{un}(\mm)$ of unweighted filtered points of $\mm$.  It is the union of the stacks $\filt(\mm)_n$ parametrising filtrations of length $n$ and an algebraic higher derived stack, locally of finite presentation over $R$ (see \cite[Proposition 2.34]{halpernleistner2025spaceaugmentedstabilityconditions}). There is the quasi compact grading morphism 
    \begin{equation}\label{eqgr}
    \operatorname{gr}:\filt^{un}(\mm)_n\ra\mm^n
    \end{equation}
    sending $$(E_1\ra\dots\ra E_n)\mapsto(E_1,\cone(E_1\ra E_2),\dots,\cone(E_{n-1}\ra E_n)),$$
    and the forgetful morphism 
    \begin{equation}\label{eqforg}
    \begin{split}\filt^{un}(\mm)&\ra\mm\\(E_1\ra\dots\ra E_m)&\mapsto E_m.
    \end{split}
    \end{equation}
    For every $n\in\Z_{\geq 0}$, the subset of the topological space $|\filt^{un}(\mm)_n|$ parameterizing HN filtrations of an object of mass $< C$  with length $n$ and phases contained in $(a, b]$ is contained in the preimage under \eqref{eqgr} of the subset of $|\mm^n|$ parameterizing $n$-tuples of semistable objects with increasing phase in $(a, b]$ and mass $< C$. Thus, this set in $|\filt^{un}(\mm)|$ is bounded by the previous discussion and so does its image in $|\mm|$ under the forgetful morphism in \eqref{eqforg}. 
    \qed\\

    From the proof above we have the following important corollary.
\begin{cor}\label{corcons}
    Suppose we have a proper stability condition $\sigma_\kappa$ on $\dd_\kappa$ for every residue field $\kappa$ of $R$, a locally constant $v:|\mm|\ra\Lambda$ and a $C$-local heart $\aaa_C$ for every DVR $C$ essentially of finite type over $R$ with the generic flatness property. Suppose also that $\aaa_C$ restricts to $\cal P_K(0,1]$ and $\cal P_{C/\mathfrak{m}_C}(0,1]$, where $K$ is the fraction field of $C$. Then there exists a unique locally constant stability condition $\sigma\in\stab(\dd/R)$ which is proper and restricts to $\sigma_K$ and $\sigma_{C/\mathfrak{m}_C}$.
\end{cor}
\begin{rmk}\label{rmkGLaction}
    As in the absolute case, the complex manifold $\stab(\dd/R)$ carries an action by $\widetilde{GL}_2^+(\R)$, the universal cover of $GL_2(\R)^+$. Recall that $\widetilde{GL}_2^+(\R)$ consists of pairs $(g,A)$, with $g:\R\ra\R$ a monotone increasing function with $g(\phi+1)=g(\phi)+1$ and $A\in GL^+(\R)$. Then, given $(g,A)\in\widetilde{GL}_2^+(\R)$ and a locally constant $\sigma$ we define 
    \begin{align*}
        A\,.Z&:=A\circ Z:\Lambda\ra\C,\\
        g\,.\cal P_{\kappa(s)}(\phi)&:=\cal P_{\kappa(s)}(g(\phi))\quad\text{for all }s\in\spec R,\\
        g\,.\cal P_C(\phi)&:=\cal P_C(g(\phi))\quad C\text{ a DVR over } R.
    \end{align*}
    Clearly, the property that the $t$-structures $(P_{\kappa(s)}(>\phi),P_{\kappa(s)}(\leq\phi))$ and $(\cal P_C(>\phi),\cal P_C(\leq\phi))$ define a sluicing on $\dd_{\kappa(s)}$ and $\dd_C$ respectively is preserved.
\end{rmk}
We emphasize another consequence of the previous theorem. Recall that by Theorem \ref{thmhlr} the condition that $$\mm_\sigma^{ss}(v)\hookrightarrow\mm_\sigma{(\phi,\phi+1]}\hookrightarrow\mm$$
are open and $\mm_\sigma^{ss}(v)$ admits a proper good moduli space is equivalent to the existence of a function $f:\R_{>0}\ra\Z_{\geq 0}$ such that 
\begin{equation}\label{eqmasshom}
\dim\Hom(G_k,E)\leq f(m(E))
\end{equation}
for all $E\in\dd_k$, where $k$ is a field of finite type over $R$ and $G$ is classical generator of $\dd$.
\begin{cor}
    If $\sigma\in\stab(\dd/R)$ verifies the mass-hom bound in \eqref{eqmasshom} so do all the $\sigma'$ in the same connected component of $\stab(\dd/R)$ containing $\sigma$.
\end{cor}
The last conclusion we are able to draw is a comparison between the datum of a proper locally constant stability condition an that of a \emph{flat family of stability conditions} as defined in \cite{bayer2021stability}. First, recall its definition within our context.
\begin{defi}[\cite{bayer2021stability}, \cite{pertusi2026noncommutative}]
    A \emph{flat family of stability conditions} on $\dd$ over $R$ is a collection of stability conditions $\ssigma=\{(Z_t,\cal P_t)\}_{t\in\spec R}$ on $\dd_t$ for every point $t\in\spec R$ such that:
    \begin{enumerate}
        \item the central charges $\{Z_t\}_{t\in\spec R}$ are locally constant in the sense of Definition \ref{defilocconstantmap},
        \item the set of points in $|\mm|$ which are $\ssigma$-semistable is open,
        \item semistable objects of a fixed class $v\in \Lambda$ are bounded,
        \item For every DVR $C$ essentially of finite type over $R$ there exists a Harder--Narasimhan structure on $\dd_C$ inducing $\sigma_K$ and $\sigma_{C/\mathfrak{m}_C}$.
    \end{enumerate}
\end{defi}
\begin{prop}\label{propcomparison}
    A proper locally constant stability condition on $\dd$ as in Definition \ref{defstabf} is equivalent to a flat family of stability conditions.
\end{prop}
\begin{proof}
    Let $\sigma$ be a proper, locally constant stability condition. In general, by \cite[Lemma 2.26]{halpernleistner2025spaceaugmentedstabilityconditions} the datum of a sluicing on $\dd_k$ for every residue field $k$ of $R$, a locally constant $v:|\mm|\ra \Lambda$ and condition (b) in Definition \ref{defstabf} imply the existence of a stability condition on $\dd_k$ for any field $k$ over $R$. This observation, together with the properness assumption on $\sigma$, grants openness of the semistable locus in $|\mm|$ and boundedness of semistable objects of a fixed class $v\in\Lambda$. Lastly, condition (a) in Definition \ref{defstabf} together with Theorem \ref{thmluicing} imply the existence of a Harder--Narasimhan structure on $\dd_C$ for any DVR $C\ra R$ essentially of finite type. 
    
    Conversely, given $\ssigma$ a flat family of stability conditions, by \cite[Proposition 2.14]{halpernleistner2025spaceaugmentedstabilityconditions} there is an induced sluicing of width 1 on $\dd_k$ for every residue field of $R$. The local constancy of the central charges translates to the existence of a locally constant $v:|\mm|\ra\Lambda$ and the support property for $\ssigma$ gives condition (b) in Definition \ref{defstabf}. Let $C\ra R$ be a DVR essentially of finite type. To a Harder--Narasimhan structure corresponds a $C$-local slicing on $\dd_C$ restricting to $\cal P_K$ and $\cal P_{C/\mathfrak{m}_C}$. Step 3 from the proof of Theorem \ref{thmmain} shows that this datum induces a sluicing on $\dd_C$ of weight one, restricting to the ones given in $\dd_K$ and $\dd_{C/\mathfrak{m}_C}$ respectively.
\end{proof}
\subsection{A possible compactification}
This last section is purely speculative, as the author does not have any supporting evidence. 

In the work \cite{halpern2023stability} the notion of a quasi-convergent path is introduced. Consider a pretriangulated dg-category $\cal C$ over a field. For $\sigma\in\stab(\cal C)$ and an object $0\neq E\in\cal C$ define the average phase as $$\bar{\phi}_\sigma(E):=\frac{1}{m_\sigma(E)}\sum_i\phi(\HN^i(E))\cdot|Z_\sigma(\HN^i(E))|,$$
where $\HN^i(E)$ denotes the $i$-th Harder--Narasimhan factor of $E$, letting us define the following function  $$\ell_\sigma(E):=\log m_\sigma(E)+i\pi\bar{\phi}_\sigma(E).$$
Given a continuous path $\sigma_\bullet:[0,+\infty)\ra\stab(\cal C)$ we say that an object $0\neq E\in\cal C$ is limit semistable if $\lim_{t\to\infty}\phi_{\sigma_t}^+(E)-\phi_{\sigma_t}^-(E)=0$. 
\begin{defi}[\cite{halpern2023stability}, Definition 2.8]
    A continuous path $\sigma_\bullet:[0,+\infty)\ra\stab(\cal C)$ is \emph{quasi-convergent} if 
    \begin{itemize}
        \item every $0\neq E\in\cal C$ has a filtration $0=E_0\ra E_i\ra\dots\ra E_{n-1}\ra E_n=E$ such that $G_i:=\cone(E_{i-1}\ra E_i)$ is limit semistable for all $i$ and $$\liminf_{t\to\infty}\phi_{\sigma_t}(G_{i+1})-\phi_{\sigma_t}(G_i)>0,$$
        \item for any pair of limit semistable objects $E$ and $F$ the limit $$\lim_{t\to\infty}\frac{\ell_{\sigma_t}(E)-\ell_{\sigma_t}(F)}{1+|\ell_{\sigma_t}(E)-\ell_{\sigma_t}(F)|}$$
        exists.
    \end{itemize}
\end{defi}
One of main results of \cite{halpern2023stability} is that from a quasi-convergent path one can extract an equivalence relation on objects of $\cal C$ inducing a corresponding semiorthogonal decomposition for the category $\cal C=\langle\cal C_1,\dots,\cal C_m\rangle$.

We propose the following generalization to the relative context. Let $\dd$ be an $R$-linear, smooth and proper category.
\begin{defi}
    A continuous path $\sigma_\bullet:[0,+\infty)\ra\stab(\dd/R)$ of locally constant stability conditions is \emph{relative quasi-convergent} if for all $s\in\spec R$ the corresponding stability conditions $\sigma_{s,\bullet}$ are quasi-convergent.
\end{defi}
\begin{rmk}
    The local constancy of the $\sigma_t$'s in $\stab(\dd/R)$ for all $t\in[0,+\infty)$ seems to impose sufficient rigidity to the notion of a relative quasi convergent path. However, it is not clear if more constraints should be added to the definition above in order for the next conjecture to hold.
\end{rmk}
\begin{conj}\label{conjaugmented}
    A relative quasi-convergent path induces an equivalence relation $\sim^i$ on $\dd$ and a corresponding $R$-linear semiorthogonal decomposition $$\dd=\langle\dd_1,\dots,\dd_m\rangle.$$ Moreover, the complex manifold $\stab(\dd/R)^\circ$ admits a compactification $\aaa\stab(\dd/R)$ by locally constant \emph{augmented} stability conditions for which the quasi-convergent paths in $\stab(\dd/R)$ are convergent. 
\end{conj}
Families of $R$-linear semiorthogonal decompositions are well understood by work of Belmans--Okawa--Ricolfi \cite{belmans2020moduli}. They show that semiorthogonal decompositions of the category of perfect complexes on a smooth and proper scheme over a base form an algebraic space over the base. In fact, it holds that given $f:\cal X\ra S$ a smooth and proper morphism with $S$ an excellent scheme, then for any integer $i\geq 2$ the functor 
\begin{equation*}
    \begin{split}
       \operatorname{SOD}_f^i:(\sch/S)&\ra(\set)\\
        T\mapsto\bigg\{&\begin{array}{c}
            \text{$T$-linear~semiorthogonal} \\
            \text{decompositions~}\Perf (\cal X_T) = \langle \aaa_{1},\dots,\aaa_i \rangle
          \end{array}\bigg\}
    \end{split}
\end{equation*}
gives an algebraic space which is étale over the base $S$. Given the conjecture above it would be natural to imagine for a smooth and proper morphism $f:\cal X\ra S$ a correspondence between relative quasi-convergent paths $\sigma_\bullet:[0,+\infty)\ra\stab(\Perf(\cal X)/S)$ and the collection of the spaces above $\{\operatorname{SOD}_f^i\}_{i\geq2}$. 
\bibliographystyle{alpha} 
\bibliography{ref}
\end{document}